\documentclass[12pt, oneside]{amsart}

\usepackage{geometry}                
\usepackage{graphicx}
\usepackage[mathscr]{euscript}
\usepackage{enumerate}
\usepackage{bbm}
\usepackage{amsthm,amssymb,amsmath}
\usepackage{subfigure}
\allowdisplaybreaks[4]
\usepackage{xcolor}

\DeclareMathOperator{\Aut}{Aut}

\begin{document}
\title{Small measure expansion in  connected noncompact groups}

\author{Jinpeng An}
\address{School of Mathematical Sciences, Peking University, Beijing, China}
\email{anjinpeng@gmail.com}

\author{Yifan Jing}
\address{Department of Mathematics, the Ohio State University, Columbus, USA}
\email{jing.245@osu.edu}

\author{Chieu-Minh Tran}
\address{Department of Mathematics, National University of Singapore, Singapore}
\email{trancm@nus.edu.sg}

\author{Ruixiang Zhang}
\address{Department of Mathematics, University of California Berkeley, CA, USA}
\email{ruixiang@berkeley.edu}

\thanks{JA was supported by the NSFC grant 11322101.}
\thanks{YJ was supported by Ben Green’s Simons Investigator Grant, ID:376201.}
\thanks{CMT was supported by Anand Pillay's NSF Grant-2054271.}
\thanks{RZ was supported by the NSF grant DMS-1856541, the Ky Fan and Yu-Fen Fan Endowment Fund at the Institute for Advanced Study and the NSF grant DMS-1926686.}


\date{} 

\newtheorem{theorem}{Theorem}[section]
\newtheorem{lemma}[theorem]{Lemma}
\newtheorem{corollary}[theorem]{Corollary}
\newtheorem{fact}[theorem]{Fact}
\newtheorem{conjecture}[theorem]{Conjecture}
\newtheorem{proposition}[theorem]{Proposition}
\theoremstyle{definition}
\newtheorem{definition}[theorem]{Definition}
\newtheorem{example}[theorem]{Example}
\newtheorem*{thm:associativity}{Theorem \ref{thm:associativity}}
\newtheorem*{thm:associativity2}{Theorem \ref{thm:associativity2}}
\newtheorem*{thm:associativity3}{Theorem \ref{thm:associativity3}}
\def\tri{\,\triangle\,}
\def\P{\mathbb{P}}

\def\E{\mathbb{E}}

\def\e{\mathbbm{1}}
\def\h{\mathrm{hdim}}
\def\ndim{\mathrm{ndim}}
\def\d{\,\mathrm{d}}
\def\dd{\mathfrak{d}}
\def\supp{\mathrm{supp}}
\def\BM{\mathrm{BM}}
\def\RR{\mathbb{R}}
\def\TT{\mathbb{T}}
\def\ZZ{\mathbb{Z}}
\def\X{\widetilde{A}}
\def\Y{\widetilde{Y}}
\def\XX{\alpha}
\def\YY{\beta}
\def\L{L^+}
\def\x{\widetilde{x}}
\def\SL{\mathrm{SL}}
\def\PSL{\mathrm{PSL}}
\def\mm{\mathrm{mod}}
\def\Stab{\mathrm{Stab}}

\newcommand\NN{\mathbb N}
\newcommand{\Case}[2]{\noindent {\bf Case #1:} \emph{#2}}
\newcommand\inner[2]{\langle #1, #2 \rangle}

\newcommand{\GL}{\mathrm{GL}}
\newcommand{\OO}{\mathrm{O}}
\newcommand{\Ad}{\mathrm{Ad}}

\newcommand{\Ker}{\mathrm{Ker}}
\newcommand{\Lg}{\mathfrak{g}}
\newcommand{\Lh}{\mathfrak{h}}
\newcommand{\Lk}{\mathfrak{k}}
\newcommand{\La}{\mathfrak{a}}
\newcommand{\Ln}{\mathfrak{n}}
\newcommand{\Lp}{\mathfrak{p}}

\begin{abstract}
Assuming $G$ is unimodular and $K>1$, and that $\mu(A^2) < K \mu(A)$. We show that this is a continuous surjective group homomorphism $\chi: G \to L$ with a compact kernel, where $L$ is a Lie group satisfying
$$\dim(L) \leq \lfloor\log K\rfloor(\lfloor\log K\rfloor+1)/2.$$ 
This dimension bound is sharp, We also establish the relationship between $A$ and its image under the quotient map, and obtain a more general version of this result for the product set $AB$ without assuming unimodularity. 

We present two applications:
\begin{enumerate}
    \item An inverse Kemperman theorem in noncompact groups. Let $G$ be a connected, noncompact, locally compact group with a left Haar measure $\mu$, and let $A$ and $B$ be nonempty, compact subsets of $G$. We classify all cases where $G$, $A$, and $B$ exhibit nearly minimal expansion properties (in the unimodular case, this means $\mu(AB)$ is close to $\mu(A) + \mu(B)$). This result completes the noncompact cases of the Kemperman inverse problem and resolves a question posed by Griesmer and Tao.
    \item An effective generalized Margulis lemma for discrete subgroups of almost-connected locally compact groups acting properly by isometries. If $B(x,4r)$ can be covered by $K$ translates of $B(x,r)$ at one point and one scale, it gives a small radius $r/2M$, and the group generated at this radius has a subgroup of index at most $M \exp((\log K)^4)$ which is nilpotent modulo a finite normal subgroup. The nilpotency step is bounded by $(\log K)^2$, and 
\[
M\leq K^{\exp(C(\log K)^4\log\log K)}
\]
and $C$ is an absolute constant. 
\end{enumerate}

\end{abstract}
\maketitle

\section{Introduction}

A central problem in additive combinatorics is to obtain descriptions of a subset $A$ of an abelian group $G$ such that  $A+A:= \{a_1+a_2: a_1, a_2 \in A\}$ has small expansion (i.e., $|A+A|< K |A|$ for a fixed constant $K$).  Freiman-type results tell us that such $A$ must be commensurable to a coset progression; the theorem for abelian groups was proven by Green and Ruzsa~\cite{GR07}, and the quantitative bound was later improved by Sanders~\cite{Sanders12}. When $K=2$, a stronger conclusion in the same line is given by Kemperman's structure theorem~\cite{KempActa}, which classifies finite subsets $A, B$ of an abelian group $G$ such that $|A+B|< |A|+|B|$. Another famous result in this direction is Freiman's $(3k-4)$ theorem, which tells us that if $A \subseteq \ZZ$ satisfies $|A+A| \leq 3|A|-4$, then $A$ must be contained in an arithmetic progression of length $2|A|-1$.
The theory have been extended to various settings, which brings together ideas from different areas
of mathematics. In this paper, we obtain the counterpart of the above results for connected noncompact locally compact groups.

Let $G$ be a {\it connected} locally compact group, $\mu_G$ a left Haar measure on $G$, and $A, B \subseteq G$ are nonempty and compact. As usual, we let $AB=\{ab:a\in A,b\in B\}$ be the productset of $A$ and $B$. We also write $A^2$ when $A=B$.
Even though we have a more general result that applies arbitrary $G$, we will assume for simplicity in the next result that $G$ is unimodular (i.e., $\mu_G$ is also a right Haar measure on $G$, which holds, for example, when  $G$ is a semisimple Lie group). 
Our main theorem tells us that small expansion in an arbitrary locally compact group arise from small expansion in a Lie group of small dimension:

\begin{theorem}[Small measure expansions in unimodular noncompact groups]\label{thm:main1.3}
Let $G$ be a connected unimodular noncompact locally compact group, $\mu_G$ is a left Haar measure. Suppose $A$ is a compact subset of $G$ of positive measure, and
\begin{equation}\label{eq: small exp}
  \mu_G(A^2)< K\mu_G(A).
\end{equation} Then there is a continuous surjective group homomorphism $\chi: G\to L$ with compact kernel, where $L$ is a connected Lie group with 
$$\dim(L) \leq \lfloor\log K\rfloor(\lfloor\log K\rfloor+1)/2.$$ Moreover, with  $A' =\chi(A)$, we have $
\mu_{L}((A')^2)\leq 64K^6\mu_{L}(A')
$.
\end{theorem}

We emphasize that the dimension bound in the above theorem is sharp. 

\begin{example}
    Let $G$ be the identity component $\mathrm{SO}^+(r,1)$ of the pseudo-orthogonal group $\mathrm{SO}(r,1)$. In this case, $\dim(G) = r(r+1)/2$. The group $G$ has $\mathrm{SO}(r)$ as a maximal compact subgroup, which has dimension $r(r-1)/2$. For every $\varepsilon>0$,  by taking a small open neighborhood $A$ of $\mathrm{SO}(r)$ in $\mathrm{SO}^+(r,1)$, the computation by the last three authors in~\cite{JTZ21} gives us $\mu_G(A^2)< 2^{r+\varepsilon} \mu_G(A)$. 
\end{example}

Theorem~\ref{thm:main1.3} has the following immediate corollary, which can be applied to works by Hrushovski and Rodriguez Fanlo~\cite{Fanlo2}, Rodriguez Fanlo~\cite{Fanlo}, and Machado~\cite{SM} on variants of small expansion problems.

\begin{corollary}\label{cor}
Let $G$ be a connected unimodular noncompact Lie group with no nontrivial compact normal subgroups. Suppose there is a compact subset $A\subseteq G$ of positive measure, and $  \mu_G(A^2)< K\mu_G(A).$ Then 
$\dim(G) \leq \lfloor\log K\rfloor(\lfloor\log K\rfloor+1)/2.$ 
\end{corollary}

An ineffective version of Theorem~\ref{thm:main1.3}, with no quantitative bound on the dimension of $L$, can be deduced from Carolino's classification of measure approximate groups of locally compact groups~\cite{CaroThesis}. When $G$ is a nilpotent Lie group with no compact normal subgroups, the bound in Corollary~\ref{cor} can be strengthened to $\lfloor\log K\rfloor$, a bound useful in the characterization of finite approximate groups; see ~\cite{Hrushovski} and also~\cite{TaoBook}. 

The main ingredient of the proof of Theorem~\ref{thm:main1.3} is the recently developed nonabelian Brunn--Minkowski inequality~\cite{JTZ21} by the last three authors. This produces a structural constraint $n-h<\log K +1$ on the group $G$, where $n-h$ is an invariant of $G$ roughly measuring its ``noncompact dimension''. A new component of this paper is a structural study of $G$ with $n-h<\log K +1$, which yields the desired  $\chi$ and $L$. Even though the Brunn--Minkowski inequality obtained in~\cite{JTZ21} is sharp for a large class of groups (e.g. matrix groups, solvable groups), it can be nonsharp for certain other groups. 
The key point is that we still manage to deduce a sharp dimension bound for all locally compact groups. This is done through a more involved analysis. The small expansion condition of $A'$ in Theorem~\ref{thm:main1.3} can be obtained using a nonunimodular generalization of an argument of Tao~\cite{Tao08} to replace $A$ with a commensurable approximate group, as well as a geometric spillover argument. 

Theorem~\ref{thm:main1.3} will be presented as a special case of Theorem~\ref{thm: asym 1.5}, which applies to all connected noncompact locally compact groups and pairs $(A,B)$ with the product set $AB$ have small measures in an appropriate sense compared to the measures of $A$ and $B$. This requires us to introduce in addition a notion of Ruzsa's distance for nonunimodular groups and develop the related machinery. We move to applications of Theorem~\ref{thm:main1.3}

\subsection{Kemperman Inverse Problem in noncompact groups}

The classic Cauchy--Davenport theorem states that if $A, B \subseteq \mathbb{Z}/p\mathbb Z$ are nonempty, then their sumset $A+B:=\{a+b : a \in A, b\in B\}$ satisfies
\(
|A + B| \geq \min \left(p,\, |A| + |B| - 1 \right).
\)
First proved by Cauchy (1813) and rediscovered by Davenport (1935), this fundamental result provides a sharp lower bound on the size of such sumsets and anticipates modern developments like Freiman's theorem. The Cauchy--Davenport's inverse problem of classifying when the equality happens and nearly happens was considered by  Vosper, Hamidoune, R\o dseth, Serra, Z\'emor~\cite{HR00, HSG06, SZ00, Vosper}.

There were subsequently many attempts to generalize and strengthen the Cauchy--Davenport theorem in different directions. A particularly far reaching generalization was obtained by Kempermann in 1964.
Consider a locally compact group $G$ equipped with a left Haar measure $\mu_G$. Let $A, B \subseteq G$ be nonempty and compact. Set $AB=\{ab:a\in A,b\in B\}$ be the productset of $A$ and $B$; we also write $A^2$ when $A=B$. When 
$G$ is \emph{unimodular} (i.e., $\mu_G$ is also a right Haar measure on $G$), Kemperman showed that 
$$ \mu_G(AB) \geq  \min\{ \mu_G(G), \mu_G(A)+\mu_G(B) -\mu_G(H) \}  $$
when $G$ has a proper compact subgroup $H$ with largest measure. The Cauchy--Davenport theorem is the special case of this when $G$ is $\ZZ_p$ and $\mu_G$ is the counting measure.

Much effort has been made toward solving the Kemperman's inverse problem. Notably, Kemperman himself obtained a famous classification of sets $A,B \subseteq G$ with $|A+B|< |A|+|B|$ when $G$ is discrete and abelian ~\cite{KempActa}. In an unpublished tour-de-force, Devos obtained a similar classification when $G$ is a discrete group~\cite{Devos}. 
On the opposite side, we have the purely continuous setting when $G$ is a  \emph{connected} unimodular locally compact group. Here, Kemperman's inequality has a particularly nice form
$$ \mu_G(AB) \geq  \min\{ \mu_G(G), \mu_G(A)+\mu_G(B) \}.  $$
With the further assumption that $G$ is abelian and compact, the conditions for equality to happen were determined by Kneser~\cite{Kneser}, and for equality to nearly happen by Tao~\cite{Tao18}. Subsequently,  Christ and Iliopoulou~\cite{ChristIliopoulou} obtained a sharp exponent version of Tao's result. Recently, the second and the third  authors~\cite{JT21} determined when equality holds in the Kemperman inequality for nonabelian and unimodular $G$.
This paper resolves the remaining case where $G$ is connected and noncompact:

\begin{theorem}[Nearly minimal expansions in noncompact unimodular groups]\label{thm:main1}
Let $G$ be a connected unimodular noncompact locally compact group, $\mu_G$ be a Haar measure. Then we have the following:
\begin{enumerate}[\rm (i)]
    \item If $\mu_G(AB)<\mu_G(A)+\mu_G(B) +\frac{1}{20}\min\{\mu_G(A),\mu_G(B)\}$, then there is a continuous surjective group homomorphism $\chi:G\to \RR$ with compact kernel, and two intervals $I,J\subseteq \RR$, such that $A\subseteq \chi^{-1}(I)$ and $B\subseteq \chi^{-1}(J)$, and $$\mu_G(\chi^{-1}(I))-\mu_G(A)<100\dd_G(A,B) \quad  \mu_G(\chi^{-1}(J))-\mu_G(B)<100\dd_G(A,B)$$
    where we set $\dd_G(A,B) = \mu_G(AB)-\mu_G(A)-\mu_G(B)$.
    \item If $\mu_G(AB)=\mu_G(A)+\mu_G(B)$, then there is a continuous surjective group homomorphism $\chi:G\to \RR$ with compact kernel, and two intervals $I,J\subseteq \RR$, such that $A= \chi^{-1}(I)$ and $B= \chi^{-1}(J)$.
\end{enumerate}
\end{theorem}

With $A=B$, our theorem tells us that if $\mu_G(A^2)< 2.05 \mu_G(A) $ then $A$ is contained in the inverse image of an interval of length $6\mu_G(A)$. This is a result along the line of  Freiman's $(3k-4)$ theorem, which asserts that if $A \subseteq \mathbb{Z}$ satisfies $|A + A| \leq 3|A| - 4$, then $A$ must be contained in an arithmetic progression of length $2|A| - 1$. In view of this, we conjecture that one can replace $1/20$ in Theorem~\ref{thm:main1} by $1$, and the same conjecture for compact abelian groups was made by Christ and Iliopoulou~\cite{ChristIliopoulou}. It is easy to see that replacing $1/20$ by a constant $>1$, the structural characterization obtained in statement (i) will no longer hold.  With more efforts, we also extend Theorem~\ref{thm:main1} to sets with measure expansions at most $4$, see Theorem~\ref{thm:main2}.

\subsection{Effective generalized Margulis lemma in almost-connected groups}
We will next discuss an application of Theorem~\ref{thm:main1.3}. Let $G$ be a locally compact Hausdorff group acting continuously and properly by isometries on a proper metric space $(X,d_X)$, and let $\Gamma\le G$ be discrete. For $x\in X$ and $s>0$, write
\[
 S_s(x)=\{\gamma\in\Gamma:d_X(x,\gamma x)<s\},
 \qquad
 \Gamma_s(x)=\langle S_s(x)\rangle.
\]
The generalized Margulis lemma of Breuillard, Green, and Tao~\cite[Corollary~11.17]{BGT} states that a packing bound gives a radius at which this subgroup is virtually nilpotent. The general argument is not effective and does not give a computable radius.

We consider the case where $G$ is \emph{almost-connected}, meaning that $G/G^0$ is compact. Here $G^0$ denotes the identity component. Theorem~\ref{thm:main1.3} gives a Lie quotient of bounded dimension. We can therefore use the effective theorem for finite linear approximate groups in place of the general approximate group theorem.

\begin{theorem}\label{thm: margulis}
There is an absolute constant $C>0$ such that the following holds. 
For any $K>1$, there is $M$ satisfies
\[
M\leq K^{\exp(C(\log K)^4\log\log K)}
\]
such that, let $G$ be an almost-connected locally compact group acting continuously and properly by isometries on a proper metric space $X$. Let $\Gamma\le G$ be discrete. Suppose that, for some $x\in X$ and $r>0$, the ball $B(x,4r)$ is covered by at most $K$ balls of radius $r$.
Then with $\Delta=\Gamma_{r/2M}(x)$, there are subgroups
$F\le E\le\Delta$ 
\begin{enumerate}[\rm (i)]
    \item $F$ is finite and $F\trianglelefteq\Delta$.
    \item $[\Delta:E]\le M\exp((\log K)^4)$.
    \item $E/F$ is nilpotent group of step at most $(\log K)^2$.
\end{enumerate}
In particular, $\Delta$ is virtually nilpotent. If $\Gamma$ is torsion-free, then $E$ itself is nilpotent of step at most $(\log K)^2$.
\end{theorem}

Note that the bounds do not depend on the number of components in $G/G^0$, and in fact the compact group $G/G^0$ may be infinite. We also remark that in the proof, $C=10^5$ suffices. 

\begin{corollary}\label{cor: homogeneous}
The conclusion of Theorem~\ref{thm: margulis} holds for every discrete subgroup $\Gamma\le\mathrm{Isom}(X)$ when $X$ is a proper connected homogeneous metric space satisfying the stated covering hypothesis.
\end{corollary}

\begin{proof}
The full isometry group acts properly on $X$. Hence it is almost connected by~\cite[Corollary~2.8]{CKLNO}. Applying Theorem~\ref{thm: margulis} finishes the proof.
\end{proof}


\subsection*{Organization of the paper}

The paper is organized as follows. Section 2 reviews the necessary background on locally compact groups. In Section 3, we introduce the Brunn--Minkowski coefficient of a group and prove a bounded-dimension Lie quotient theorem for groups with uniformly bounded Brunn--Minkowski coefficients. Section 4 establishes quotient domination results for sets with small measure expansions. In Section 5, we study minimal doubling sets in nonunimodular groups. Section 6 contains the proofs of Theorem~\ref{thm:main1.3} and the noncompact Kemperman Inverse Problem. In Section 7, we prove the effective generalized Margulis lemma. Finally,  Appendix A develops the nonunimodular version of Ruzsa's distance and the associated machinery.

\section{Preliminaries}

Throughout the paper, $G$ is a connected locally compact group, $\mu_G$ is a left Haar measure and $\nu_G$ is a right Haar measure on $G$. When the ambient group is clear in the context, we sometimes simply use $\mu$ and $\nu$ to denote the left and the right Haar measures respectively.  All the $\log$ in this paper are logarithms with base $2$. We say $G$ is \textbf{unimodular} if a left Haar measure is also a right Haar measure on $G$.

The structure of locally compact groups was characterized by Gleason~\cite{Gleason} and Yamabe~\cite{Yamabe}, which resolved Hilbert's 5th problem:
\begin{fact}[Gleason--Yamabe Theorem]\label{fact: gleason}
Suppose $G$ is a locally compact group. Then there is an open subgroup $G'$ of $G$, such that for any open neighborhood $U$ of $G'$ containing identity, there is a compact normal subgroup $H\subseteq U$, such that $G'/H$ is a Lie group.
\end{fact}

Given a locally compact group $G$, and $\mu$ is a left Haar measure on $G$. Define $\mu_x(A)=\mu(Ax)$ for every $x\in G$ and every measurable sets $A$.
For every $x\in G$, the modular function is defined by $\Delta_G: x \mapsto \mu_x/\mu$.
In particular when the image of $\Delta_G$ is always $1$, $G$ is unimodular. In general, $\Delta_G$ takes
values in $\RR^{>0}$. The following fact from~\cite{BourbakiTopology} records some
basic properties of the modular function:
\begin{fact}\label{fact: 2.2}
Let $G$ be a locally compact group, $\mu$ a left Haar measure and $\nu$ a right Haar measure on $G$.
\begin{enumerate}[\rm (i)]
    \item Suppose $H$ is a normal closed subgroup of $G$, then $\Delta_H=\Delta_G|_H$. In particular, if $H=\ker\Delta_G$, then $H$ is unimodular.
    \item The function $\Delta_G: G\to \RR^{>0}$ is a continuous homomorphism.
    \item For every $x\in G$ and every measurable set $A$, we have $\mu(Ax)=\Delta_G(x)\mu(A)$, and $\nu(xA)=\Delta_G^{-1}(x)\nu(A)$.
    \item There is a constant $c$ such that $\int_G f\d \mu = c \int_G f\Delta_G\d \nu$ for every $f\in C_c(G)$.
\end{enumerate}
\end{fact}

Let $\phi$ be an automorphism on $G$. Clearly, $\phi^{-1}(\mu)$ is also a left Haar measure on $G$. The number $\phi^{-1}(\mu)/\mu$ is called \textbf{modulus of the automorphism} $\phi$, and is denoted by $\mm_G(\phi)$. We use the following fact from~\cite[Proposition 2.7.10]{BourbakiInt}

\begin{fact}\label{fact: auto}
Let $G$ be a locally compact group, $H$ be a closed normal subgroup of $G$, and $\pi: G\to G/H$ is the quotient map. Let $u$ be  an automorphism
of $G$ with $u(H)=H$, $u'$ is the restriction of $u$ on $H$, and $u''$ is the projection of $u$ on $G/H$. Then
\[
\mm_G(u)=\mm_H(u')\cdot \mm_{G/H}(u'').
\]
\end{fact}

We use the following integral formula ~\cite[Theorem 2.49]{Folland} in our proofs.
\begin{fact}[Quotient integral formula]\label{fact: Quotient Integral Formula}
Let $G$ be a locally compact group, and let $H$ be a closed normal subgroup of $G$. Given $\mu_G$, $\mu_H$ left Haar measures on $G$ and on $H$. Then  there is a unique left Haar measure $\mu_{G/H}$ on $G/H$, such that
for every $f\in C_c(G)$,
\[
\int_G f(x)\d \mu_G(x)=\int_{G/H}\int_H f(xh) \d \mu_H(h)\d \mu_{G/H}(x).
\]
\end{fact}

The following well-known fact is the Freiman $3k-4$ theorem over $\RR$.

\begin{fact}[$3k-4$ theorem in $\RR$]\label{fact: inverse R}
Suppose $A,B$ are compact subsets of $\RR$, and $\lambda$ is the Lebesgue measure on $\RR$. If
\[
\lambda(A+B)<\lambda(A)+\lambda(B)+\min\{\lambda(A),\lambda(B)\},
\]
then there are compact intervals $I,J\subseteq \RR$, with
\[
\lambda(I)\leq\lambda(A+B)-\lambda(B),\quad \lambda(J)\leq\lambda(A+B)-\lambda(A),
\]
and $A\subseteq I$, $B\subseteq J$.
\end{fact}

\section{Small expansions and low dimension Lie quotients}

Let $G$ be a noncompact group. Given two compact sets $A,B\subseteq G$ with positive measures, the {\bf Brunn--Minkowski coefficient} of $(A,B)$ is the real number  $r\geq 0$ such that
\[
\left(\frac{\nu(A)}{\nu(AB)}\right)^{\frac{1}{r}}+\left(\frac{\mu(B)}{\mu(AB)}\right)^{\frac{1}{r}}= 1.
\]
In this case, we write $\BM(A,B)=r$. Assuming that $G$ is unimodular, we have
\[
\mu(AB)^{1/r} = \mu(A)^{1/r} +\mu(B)^{1/r}.
\]
If, moreover, $A=B$, then $\mu(A^2) = 2^r \mu(A)$. Thus, we can think of $r$ as a measuring of the expansion of the pair $(A,B)$ where we allow $A$ and $B$ to have different size.

The  {\bf Brunn--Minkowski coefficient} of $G$ is the largest nonnegative real number $r$ such that
$$\BM(A,B) \geq r \quad
\text{ for every compact } A,B\subseteq G \text{ with positive measures}.$$
In this case, we  write $\BM(G) = r$.
The classical  Brunn--Minkowski inequality for $\RR^d$ tells us that $\BM(\RR^d)=d$.

The problem of determining the Brunn--Minkowski for general locally compact group was studied by the second, third, and fourth author in~\cite{JTZ21}. Toward stating our conjectural answer, we define the {\bf noncompact Lie dimension} $\ndim(G)$ of $G$ as follows. Using Theorem~\ref{fact: gleason}, one can choose an open subgroup $G'$ of $G$ and a normal compact subgroup $H$ of $G'$
such that $G'/H$ is a Lie group. Set
\[
\ndim(G) = \dim(G'/H) - \max\{ \dim(K) : K \text{ is a compact subgroup of } G'/H\}.
\]
It was verified in~\cite{JTZ21} that  $\ndim(G)$ is well-defined (i.e., it is independent of the choice of $G'$). The following properties of $\ndim(G)$, proven in~\cite{JTZ21}, will be used later on.

\begin{fact}\label{fact: nLdim}
Let $G$ be a locally compact group. Then
\begin{enumerate} [\rm (i)]
    \item If $G$ is compact, then $\ndim(G)=0$.
    \item Suppose $1\to H\to G\to G/H\to 1$ is a short exact sequence of connected Lie groups. Then $\ndim(G)=\ndim(H)+\ndim(G/H)$.
\end{enumerate}
\end{fact}

Below is the conjecture proposed in~\cite{JTZ21}; it was also proven there that it suffices to verify the conjecture for simply connected simple Lie groups.

\begin{conjecture}[Nonabelian Brunn--Minkowski Conjecture]\label{conj: main}
If $G$ is a locally compact group with noncompact Lie dimension $n$, then $\BM(G)=n$.
\end{conjecture}

We next review the result toward Conjecture~\ref{conj: main} proven in~\cite{JTZ21}; this will be a central ingredient in our later proof. We have the following fact of Lie groups, which is a consequence of~\cite[Proposition 5.4.3]{hilgert2011structure}:

\begin{fact} \label{fact: radical}
Suppose $\mathfrak{g}$ is a Lie algebra. Then $\mathfrak{g}$ has a largest solvable 
ideal $\mathfrak{q}$.
If $G$ is a Lie group with Lie algebra $\mathfrak{g}$ and $\exp: \mathfrak{g} \to G$ is the exponential map, then $Q =\langle\exp(\mathfrak{q})\rangle$ is the largest closed connected solvable normal subgroup of $G$. Hence, $Q$ is a characteristic subgroup of $G$.
\end{fact}

We call such $Q$ in Fact~\ref{fact: radical} the \emph{radical} of $G$. A Lie group is \emph{semisimple} if its Lie algebra is semisimple, or equivalently, if it has trivial radical.  The following fact is also a consequence of~\cite[Proposition 5.4.3]{hilgert2011structure}:

\begin{fact}\label{fact: semisimple} Let $G$ be a  connected Lie group.
Let $Q$ be the radical of $G$. Then $S= G/Q$ is a semisimple Lie group.
\end{fact}

Now we are going to define the {\bf helix dimension} of a locally compact group $G$, denoted by $\h(G)$. Let $G'/H$ be the Lie group obtained from Theorem~\ref{fact: gleason}, and  let $Q$ be the radical of $G'/H$. Note that $(G'/H)_0/Q)$ has discrete center $Z((G'/H)_0/Q)$. We set
\[
\h(G) = \mathrm{rank}(Z((G'/H)_0/Q)).
\]
In~\cite{JTZ21}, $\h(G)$ is well-defined, i.e., it is independent of the choice of $G'$. In general, as shown in~\cite{JTZ21}, $\h(G)$ is not split with respect to short exact sequences. However, we have the following fact which is proven in~\cite{JTZ21}.

\begin{fact}\label{fact: helix}
Let $G$ be a locally compact group. Then
\begin{enumerate}[\rm (i)]
    \item Suppose $1\to H\to G\to G/H\to 1$ is a short exact sequence of connected locally compact groups, and $H$ is compact. Then $\h(G)=\h(G/H)$.
    \item Suppose $1\to H\to G\to \RR\to 1$ is a short exact sequence of locally compact groups. Then $\h(G)=\h(H)$.
    \item Let $G$ be a connected Lie group of noncompact dimension $n$ and helix dimension $h$. Then $n\geq 3h$. If $G$ is semisimple, and $n>0$, then $n \geq 2$.
\end{enumerate}
\end{fact}

The following fact is the main result of~\cite{JTZ21}.
\begin{fact}[Jing--Tran--Zhang]\label{fact: BM}
Let $G$ be a locally compact group such that $\ndim(G)=n$ and $\h(G)=h$. Then,
\[
n-h \leq \BM(G)  \leq n.
\]
Moreover, solvable locally compact groups and real algebraic groups has helix dimension $0$, so this verifies the nonabelian Brunn--Minkowski conjecture for these groups.
\end{fact}

We will also need the following facts about Lie groups.

\begin{fact} \label{Liegroupsfacts} We have the following:
\begin{enumerate}[\rm (i)]
    \item Any two maximal compact subgroups of a connected Lie group are conjugate.
    \item A maximal compact subgroup of $\GL(n, \RR)$ is the orthogonal group $O(n)$, which has dimension 
    $n(n-1)/2$.
\end{enumerate}
\end{fact}

We now prove the following key lemma.

\begin{lemma} \label{lem: characteristicnormalsubgroup}
Let $G$ be a connected locally compact group  with noncompact dimension $n$ {\color{black} and helix dimension $h$}. Then $G$ has a compact connected {\color{black} characteristic} subgroup $H$ such that $G/H$ 
{\color{black} is a Lie group of dimension at most $(n-h)(n-h+1)/2$}.
\end{lemma}

\begin{proof}
By the Gleason--Yamabe theorem (Fact~\ref{fact: gleason}), we can assume that $G$ is a connected Lie group. 
{\color{black} Let us divide the proof into steps.
\medskip

\textbf{Step 1.} We first prove the weaker statement that there is a compact connected characteristic subgroup $H$ of $G$ such that $\dim G/H\le n(n+1)/2$.}
Let $K$ be a maximal compact subgroup of $G$, and denote the Lie algebras of $G$ and $K$ by $\Lg$ and $\Lk$,
respectively. Let $\Ad_K:K\to\GL(\Lg)$
be the adjoint representation of $K$ {\color{black} on $\Lg$. Since $K$ is compact, there is an $\Ad_K$-invariant inner product on $\Lg$, and the orthogonal complement of $\Lk$ in $\Lg$, denoted by $\Lp$, is $\Ad_K$-invariant. }
Let
$$\rho:K\to\GL(\Lp)$$ be the {\color{black} subrepresentation}
of $\Ad_K$, and $H:=\Ker(\rho)^\circ$ be the identity component of the $\Ker(\rho)$. We will show that the compact group $H$ satisfies the requirement.

First, the space $\Lp$ has dimension $n$, so by Fact~\ref{Liegroupsfacts}(ii), we have
$$\dim G/H=\dim G/K+\dim K/H=n+\dim\rho(K)\le n+n(n-1)/2=n(n+1)/2.$$ 

Next, we prove that $H$ is normal in $G$. {\color{black} Since $H$ is connected, it suffices to show that the Lie algebra $\Lh$ of $H$ is an ideal of $\Lg$. Clearly, $\Lh$ is an ideal of $\Lk$. On the other hand, since $H$ acts trivially on $\Lp$, we have $[\Lh,\Lp]=0$. Thus $[\Lh,\Lg]=[\Lh,\Lk\oplus\Lp]=[\Lh,\Lk]+[\Lh,\Lp]\subset\Lh$. Hence $\Lh$ is an ideal of $\Lg$.}

Finally, we show that $\sigma(H)=H$ for every $\sigma\in\Aut(G)$. Since $\sigma(K)$ is again a maximal compact subgroup of $G$, there exists $g\in G$ such that $c_g(K)=\sigma(K)$, where $c_g$ is the inner automorphism of $G$ induced by $g$. Since the automorphism $\tau:=c_{g^{-1}}\circ\sigma$ satisfies $\tau(K)=K$, it follows from the definition of $H$ that $\tau(H)=H$. Thus $\sigma(H)=c_g(\tau(H))=c_g(H)=H$. 
{\color{black} This proves that $H$ is a characteristic subgroup of $G$. }
\medskip

{\color{black} 
\textbf{Step 2.} We now prove the lemma for the case that $G$ is semisimple without nontrivial connected compact characteristic subgroups. In this case, $G/Z(G)$ is again semisimple without nontrivial connected compact characteristic subgroups. Denote $n'=\ndim(G/Z(G))$. Then $n=n'+\mathrm{rank}(Z(G))$, $h=\mathrm{rank}(Z(G))$. It follows that $n-h=n'$. By Step 1, we have
$$\dim G=\dim G/Z(G)\le\frac{n'(n'+1)}{2}=\frac{(n-h)(n-h+1)}{2}.$$
This proves the lemma for such $G$.
\medskip

\textbf{Step 3.} For the general case, let $Q$ be the radical of $G$, $K\subseteq G/Q$ be the maximal connected compact characteristic subgroup of $G/Q$, and $R\subseteq G$ be the preimage of $K$ under the quotient homomorphism $G\to G/Q$. Then $R$ is a connected characteristic closed subgroup of $G$, and $G/R\cong(G/Q)/K$ is semisimple without nontrivial connected compact characteristic subgroups. Denote $n_1=\ndim(R)$, $n_2=\ndim(G/R)$, $h_2=\h(G/R)$.
By Steps 1 and 2, there is a connected compact characteristic subgroup $H$ of $R$ such that
$$\dim R/H\le\frac{n_1(n_1+1)}{2},$$
and we have
$$\dim G/R\le\frac{(n_2-h_2)(n_2-h_2+1)}{2}.$$
The group $H$ is also characteristic in $G$, and we have $n=n_1+n_2$, $h=h_2$. It follows that
\begin{align*}
\dim G/H=\dim R/H+\dim G/R&\le\frac{n_1(n_1+1)}{2}+\frac{(n_2-h_2)(n_2-h_2+1)}{2}\\
&\le\frac{(n-h)(n-h+1)}{2}.
\end{align*}
This completes the proof. }
\end{proof}

 We now prove the first main result of this section.

\begin{theorem} \label{thm: lowdimensionquotient}
Suppose $G$ is a connected locally compact group with $\BM(G) =r$. Then there is a connected compact normal subgroup $H$ of $G$ such that  $G/H$
is a Lie group of dimension at most $r(r+1)/2$. 
\end{theorem}

\begin{proof}
This is a combination of Fact~\ref{fact: BM} and Lemma~\ref{lem: characteristicnormalsubgroup}. 
\end{proof}

The following facts~\cite{Jacobson} come from the classification of Lie algebra of dimensions at most $2$. 
\begin{fact} \label{fact: classification123}
Suppose $G$ is a simply connected Lie group. We have the following:
\begin{enumerate}[\rm (i)]
    \item if $\dim(G)=1$, then $G=\RR$;
    \item if $\dim(G)=2$, then $G$ is either $\RR^2$ or the affine group of the line $\text{Aff}(1)$.
\end{enumerate}
\end{fact}

\begin{lemma} \label{lem: n-h=1,2}
Let $G$ be a connected locally compact group with noncompact dimension $n$ and helix dimension $h$. Then we have the following:
\begin{enumerate}[ \rm (i) ]
    \item If  $n-h=1$, then there is a connected compact normal subgroup $H$ of $G$ such that $G/H$ is $\RR$.
    \item If  $n-h=2$, then there is a connected compact normal subgroup $H$ of $G$ such that $G/H$ is either $\RR^2$, the affine group of the line, a semidirect product $\RR^2 \rtimes \TT$, or a cover of  $\PSL_2(\RR)$.
\end{enumerate}
\end{lemma}

\begin{proof}
To prove (i), we need only to notice that by Lemma \ref{lem: characteristicnormalsubgroup}, there is a connected compact normal subgroup $H$ of $G$ such that $G/H$ is a connected Lie group with $\dim G/H\le 1$. Since $n>0$, $G/H$ is noncompact, hence is isomorphic to $\RR$.

We now prove (ii). In view of Lemma \ref{lem: characteristicnormalsubgroup}, Facts~\ref{fact: nLdim}(ii) and~\ref{fact: helix}(i), and quotienting out the maximal connected compact normal subgroup of $G$ if necessary, we may assume that $G$ is a connected Lie group without nontrivial connected compact normal subgroups such that $\dim G\le3$ and $n-h=2$. Note that $\dim G\ge n\ge2$. So $\dim G=2$ or $3$. If $\dim G=2$, then $n=2$, and it follows from the classification of $2$-dimensional Lie algebras (Fact~\ref{fact: classification123}(ii))
that $G$ is either $\RR^2$ or the affine group of the line. Assume that $\dim G=3$. Then $G$ is either semisimple or solvable. If $G$ is semisimple, since it is noncompact, it must be a cover of $\PSL_2(\RR)$. Assume that $G$ is solvable. Then $h=0$, thus $n=2$. It follows that $G$ has a $1$-dimensional maximal compact subgroup, say $K$, which must be isomorphic to $\TT$. As in the proof of Lemma  \ref{lem: characteristicnormalsubgroup}, we take an $\Ad_K$-invariant subspace $\Lp$ of $\Lg$ complementary to $\Lk$, and consider the subrepresentation $\rho:K\to\GL(\Lp)$ of $\Ad_K$. Let us identify $K$ with $\TT=\{e^{i\theta}:\theta\in[0,2\pi)\}$. Since $\dim\Lp=2$, there exist an integer $k$ and a basis $\{X,Y\}$ of $\Lp$ such that the matrix of $\rho(e^{i\theta})$ is \[\begin{pmatrix}
\cos k\theta & -\sin k\theta \\ \sin k\theta & \cos k\theta
\end{pmatrix}.\] It follows that there is a nonzero vector $Z\in\Lk$ such that $[Z,X]=kY$, $[Z,Y]=-kX$. If $k=0$, then $\Lk$ is an ideal of $\Lg$, which implies that $K$ is normal in $G$, contradicting to the assumption that $G$ has no nontrivial connected compact normal subgroups. So $k\ne0$. It follows that $\Lp\subset[\Lg,\Lg]$. Since $\Lg$ is solvable, we must have $[\Lg,\Lg]=\Lp$. Suppose $[X,Y]=aX+bY$, where $a,b\in\RR$. By the Jacobi identity, we have
\begin{align*}
0&=[X,[Y,Z]]+[Y,[Z,X]]+[Z,[X,Y]]\\
&=[X,kX]+[Y,kY]+[Z,aX+bY]=akY-bkX.
\end{align*}
So $a=b=0$. It follows that $\Lp$ is an abelian ideal of $\Lg$. Thus $\Lg=\Lp \rtimes \Lk$. This, together with the fact that $K$ is maximal compact in $G$, implies that $G$ is isomorphic to a semidirect product $\RR^2 \rtimes \TT$.
\end{proof}

The next theorem is the second main result of the section.

\begin{theorem}\label{thm: BM small}
Let $G$ be a connected noncompact locally compact group, and let $A,B$ be compact subsets of $G$. Then we have the following:
\begin{enumerate}[\rm (i)]
    \item If
$\BM(A,B)<2$,
then there is a continuous surjective group homomorphism with compact connected kernel mapping $G$ to $\RR$.
\item If $\BM(A,B)<3$, then there is a continuous surjective group homomorphism with compact connected kernel mapping $G$ to either $\RR$, $\RR^2$, $\RR^2 \rtimes \TT$, the affine group of the line, or a  cover of  $\PSL_2(\RR)$.
\end{enumerate}
If the nonabelian Brunn--Minkowski conjecture holds, we can strengthen the last case of (ii) into a finite cover of 
 $\PSL_2(\RR)$.
\end{theorem}

\begin{proof}
 Statements (i) and (ii) follow from Fact~\ref{fact: BM} and Lemma~\ref{lem: n-h=1,2}.
If the nonabelian Brunn--Minkowski conjecture holds and the last case of (ii) is an infinite cover of $\PSL_2(\RR)$ (which must be $\widetilde{\SL_2(\RR)}$), then $\BM(A,B)\ge\BM(G)=\ndim(G)=\ndim(\widetilde{\SL_2(\RR)})=3$, a contradiction. 
\end{proof}

\section{Quotient Dominations}
In this section, we assume there is a short exact sequence
\[
1\to H\to G\to G/H\to 1,
\]
where $H$ is a compact subgroup of $G$.
We first consider the case when sets $(A,B)$ has minimal measure expansion on $G$. In this case, we will have a quotient domination result.

\begin{theorem}[Quotient dominations]\label{thm: quo domi}
Let $G$ be a connected unimodular noncompact group, and $\BM(G)=n$. Let $A,B$ be compact subsets in $G$ such that
\[
\mu_G(AB)^{\frac1n}=\mu_G(A)^{\frac1n}+\mu_G(B)^{\frac1n}.
\]
Suppose $\pi: G\to G/H$ is the quotient map. Then there are compact sets $A',B'$ in $G/H$ with 
\[
\mu_{G/H}(A'B')^{\frac1n}=\mu_{G/H}(A')^{\frac1n}+\mu_{G/H}(B')^{\frac1n},
\]
and $A= \pi^{-1}(A')$, $B= \pi^{-1}(B')$.
\end{theorem}

\begin{proof}
Define the fiber length function $G/H\to \RR$:
\[
f_A(g)=\mu_H(g^{-1}A\cap H).
\]
We similarly define fiber length functions for $B$ and for $AB$. We use $L^+_f$ to denote the superlevel set of $f$, that is
\[
L^+_f(t)=\{x\in G: f(x)\geq t\}. 
\]
From the fact that $\mu_G$ can be pushed forward to a Haar measure on $G/H$, one can see that $\BM(G/H) \geq  n$. It is clear that $\BM(H)=0$. Hence, 
\begin{equation}\label{eq: G when n_1=0}
 \mu_{G/H}^{1/n}\left(L^{+}_{f_{AB}}\left(\max\{t_1,t_2\}\right)\right)\geq   \mu_{G/H}^{1/n}\left( L^+_{f_A}(t_1) \right)+ \mu_{G/H}^{1/n}\left( L^+_{f_B}(t_2) \right).
\end{equation}
Given a compact set $\Omega$ in $G$, we define
\[
E_\Omega(t)=\mu_{G/H}^{1/n}(L^+_{f_\Omega}(t)).
\]
Thus by \eqref{eq: G when n_1=0}, we have $E_{AB}(\max\{t_1,t_2\})\geq E_A(t_1)+E_B(t_2)$ for all $t_1,t_2 \in \RR$. It then follows that 
\begin{equation}\label{relationofmuR}
\mu_\RR(L^+_{E_{AB}}({\color{black}s}_1+{\color{black}s}_2))\geq \max\{\mu_\RR(L^+_{E_{A}}({\color{black}s}_1)),\mu_\RR(L^+_{E_{B}}({\color{black}s}_2))\}.
\end{equation}
Now we consider $\mu_G(AB)$. We have
\begin{align}\label{eq: the equation below 21}
    \mu_G^{1/n}(AB)
    =\left(\int_{\RR^{>0}}E_{AB}^{n}(s)\d s \right)^{1/n}
    =\left(\int_{\RR^{>0}} n{\color{black}s}^{n-1}\mu_{\RR}(L_{E_{AB}}^+({\color{black}s}))\d {\color{black}s} \right)^{1/n}
\end{align}
Note that $\sup_t E_A(t)=\mu_{G/H}(\pi A)$ and $\sup_t E_B(t)=\mu_{G/H}(\pi B)$.
Then $n^{-1/n}\mu_G^{1/n}(AB)$ is at least 
\[
\left((\mu_{G/H}(\pi A)+\mu_{G/H}(\pi B))^{n}\int_0^1 s^{n-1}\mu_{\RR}(L_{E_{AB}}^+(\mu_{G/H}(\pi A)+\mu_{G/H}(\pi B))s )\d s\right)^{\frac1n}.
\]
Using \eqref{relationofmuR}, the above quantity is at least
\begin{align*}
  &\,\Bigg((\mu_{G/H}(\pi A)+\mu_{G/H}(\pi B))^{n}\\
   &\quad\cdot \max\left\{  \int_0^1 {\color{black}s}^{n-1} \mu_{\RR}(L^+_{E_A}(\mu_{G/H}(\pi A) {\color{black}s})\d {\color{black}s},\int_0^1 {\color{black}s}^{n-1} \mu_{\RR}(L^+_{E_B}(\mu_{G/H}(\pi B) {\color{black}s}) \d {\color{black}s}\right\}\Bigg)^{1/n}  
\end{align*}
Now by H\"older's inequality, the above quantity is at least 
\begin{align*}
    &\, \left(\mu_{G/H}(\pi A)^{n}\int_0^1 {\color{black}s}^{n-1}\mu_{\RR}(L^+_{E_A}(\mu_{G/H}(\pi A) {\color{black}s}))\d {\color{black}s} \right)^{1/n} \\
    &\quad + \left(\mu_{G/H}(\pi B)^{n}\int_0^1 {\color{black}s}^{n-1}\mu_{\RR}(L^+_{E_B}(\mu_{G/H}(\pi B) {\color{black}s}))\d {\color{black}s} \right)^{1/n}.
\end{align*}
It is easy to check the above quantity is equal to
\[
n^{-1/n}\mu_G^{1/n}(A) + n^{-1/n}\mu_G^{1/n}(B). 
\]

Since we have that $\mu_G(AB)^{1/n}=\mu_G(A)^{1/n}+\mu_G(B)^{1/n}$, equalities hold everywhere in the above inequalities. This in particular implies that
\[
\mu_{G/H}(\pi A\pi B)^{\frac1n}=\mu_{G/H}(\pi A)^{\frac1n}+\mu_{G/H}(\pi B)^{\frac1n},
\]
and $\mu_G(A)=\mu_{G/H}(\pi A)$, $\mu_G(B)=\mu_{G/H}(\pi B)$. By the compactness assumption of $A$ and $B$, we have $A=\pi^{-1}\pi A$ and $B=\pi^{-1}\pi B$.
\end{proof}


We introduce the notion of discrepancy, which is standard in the literature. 
\begin{definition}[discrepancy]
Let $G$ be a unimodular group, and $\mu_G$ a Haar measure on $G$. Given compact sets $A,B\subseteq G$, the \emph{discrepancy} of $A$ and $B$, denoted by $\dd_G(A,B)$, is $\mu_G(AB)-\mu_G(A)-\mu_G(B)$.
\end{definition} 

When $\dd_G(A,B)$ is small, we use the following theorem~\cite[Theorem 6.5]{JT21} by the middle two authors.

\begin{theorem}[Almost quotient dominations]\label{thm: almost domi}
Let $G$ be a connected noncompact unimodular group, and $A,B$ be compact subsets in $G$ such that  $$\dd_G(A,B)<\min\{\mu_G(A),\mu_G(B)\}.$$  Then there are compact sets $A',B'$ in $G/H$ with
\[
\dd_{G/H}(A',B')<7\dd_G(A,B),
\]
and $\mu_G(A\tri \pi^{-1}(A'))\leq 3\dd_G(A,B)$, $\mu_G(B\tri \pi^{-1}(B'))\leq 3\dd_G(A,B)$.
\end{theorem}

Let $\pi:G\to G/H$ be the quotient map. For an open approximate group $A$ on $G$, one can easily see that $\pi A$ is also an open approximate group on $G/H$. The next proposition show that small expansion sets have a similar behaviour.

\begin{proposition}\label{prop: almost domi sym}
Let $G$ be a connected noncompact unimodular group, and $A$ a compact subset of $G$ with positive measure. Suppose $\mu_G(A^2)\leq K\mu_G(A)$, and $\pi:G\to G/H$ is the quotient map. Then $\mu_{G/H}(\pi A^2)\leq 64K^6\mu_{G/H}(\pi A)$.
\end{proposition}

\begin{proof}
Let $f_A:G/H\to \RR$ be the fiber length function that $f_A(g)=\mu_H(g^{-1}A\cap H)$. Set $\alpha=\sup_g f_A(g)$, clearly $\alpha\leq 1$. Set $N=2K$. Define
\[
\pi A_1=\Big\{g\in G/H: f_A(g)\geq\frac{\alpha}{N}\Big\}, \quad\text{and}\quad \pi A_2=\pi A\setminus \pi A_1.
\]
Note that $K\mu_G(A)\geq \mu_G(A^2)\geq \alpha \mu_{G/H}(\pi A)\geq \mu_G(A)+(1-1/N)\alpha \mu_{G/H}(\pi A_2)$. Hence,
\begin{align*}
\mu_{G/H}(\pi A_2)&\leq \frac{K-1}{\alpha}\frac{N}{N-1}\mu_G(A)\\
&\leq \frac{K-1}{\alpha}\frac{N}{N-1}\left(\alpha\mu_{G/H}(\pi A_1)+\frac{\alpha}{N}\mu_{G/H}(\pi A_2)\right).
\end{align*}
By the choice of $N$, we have
\[
\mu_{G/H}(\pi A_2)\leq 2(K-1)\mu_{G/H}(\pi A_1).
\]

On the other hand, we have
\begin{align*}
    \mu_{G/H}(\pi A_1 \pi A_1)&\leq \frac{N}{\alpha}\mu_G(A^2)\\
    &\leq \frac{KN}{\alpha}\left(\alpha\mu_{G/H}(\pi A_1)+\frac{\alpha}{N}\mu_{G/H}(\pi A_2)\right)\leq 2K(2K-1)\mu_{G/H}(\pi A_1).
\end{align*}
For $\mu_{G/H}(\pi A\pi A_1)$ and $\mu_{G/H}(\pi A_1\pi A)$, we use similar ideas to get
\begin{align*}
   \max\{ \mu_{G/H}(\pi A \pi A_1), \mu_{G/H}(\pi A_1 \pi A)\}&\leq  \frac{N}{\alpha}\mu_G(A^2)\\
   &\leq 2K(2K-1)\mu_{G/H}(\pi A_1)\\
   &\leq 2K(2K-1)(\mu_{G/H}(\pi A_1)\mu_{G/H}(\pi A))^{\frac12}.
\end{align*}
Let $d$ be the Ruzsa distance defined in~\eqref{eq: dist}. Hence $d(\pi A_1, (\pi A_1)^{-1})\leq 2\log 2K(2K-1)$, and $\max\{d(\pi A_1, (\pi A)^{-1}), d(\pi A, (\pi A_1)^{-1}) \}\leq 2\log 2K(2K-1)$. Using Lemma~\ref{lem: ruzsa},
\begin{align*}
d(\pi A, (\pi A)^{-1})&\leq d(\pi A, (\pi A_1)^{-1})+ d((\pi A_1)^{-1}, \pi A_1)+d(\pi A_1, (\pi A)^{-1}) \\
&\leq 2\log 8K^3(2K-1)^3.
\end{align*}
This implies $\mu_{G/H}(\pi A^2)\leq 64K^6 \mu_{G/H}(\pi A)$.
\end{proof}

It might be of interesting to sharpen the constant $64K^6$ in the above lemma.

\section{Minimal expansions in nonunimodular  groups}

In this section, we will study the sets with minimal expansions in nonunimodular groups. Suppose $G$ is a nonunimodular group with $\BM(G)=2$, the Brunn--Minkowski inequality given in~\cite{JTZ21} asserts that for any compact sets $A, B$ one has
\[
\left(\frac{\nu_G(A)}{\nu_G(AB)}\right)^{\frac12} + \left(\frac{\mu_G(B)}{\mu_G(AB)}\right)^{\frac12}\leq 1. 
\]
Our main result in the section is the following theorem, which shows that the equality in the above inequality cannot hold under certain conditions.

\begin{theorem}\label{thm: ax+b}
Let $G$ be a nonunimodular connected locally compact group, and $A$ is a compact subset of $G$ with positive measure. Then
\[
\left(\frac{\nu_G(A)}{\nu_G(A^2)}\right)^{\frac12} + \left(\frac{\mu_G(A)}{\mu_G(A^2)}\right)^{\frac12}<1.
\]
\end{theorem}

\begin{proof}
Note that $G$ is nonunimodular. Let $H=\ker\Delta_G$, then $G/H\cong \RR^{>0}$. As $G$ is nonunimodular, $H$ is noncompact. Note also that $H$ is unimodular. Let $\mu_\RR$ be a Haar measure on $\RR^{>0}$ with respect to multiplication. 
Define the right and left fiber length functions $G/H\to\RR$:
\[
r_A(g)=\mu_H(Ag^{-1}\cap H), \quad \ell_A(g )=\mu_H(g^{-1}A\cap H).
\]
In particular, we have 
\begin{equation}\label{eq: relation between r and l}
   \ell_A(g)=\Delta_G(g)^{-1}r_A(g). 
\end{equation}
We similarly define the  right and left fiber length function for   $A^2$.
Note that for compact sets $X_1,X_2\subseteq H$, and $g_1,g_2\in G$, $X_1g_1g_2X_2\subseteq g_1g_2 H =H g_1g_2$ as $H$ is normal. By Kemperman's inequality on $H$ and Fact~\ref{fact: auto}, we have
\begin{align}
\mu_H((g_1g_2)^{-1}X_1g_1g_2X_2)&\geq \mu_H((g_1g_2)^{-1}X_1g_1g_2)+\mu_H(X_2) \nonumber\\
&=\Delta_G((g_1g_2)^{-1})\mu_H(X_1)+\mu_H(X_2), \label{eq: fibre product 1}
\end{align}
and similarly, 
\begin{align}
\mu_H(X_1g_1g_2X_2(g_1g_2)^{-1})&\geq \mu_H(X_1)+\mu_H(g_1g_2X_2(g_1g_2)^{-1}) \nonumber\\
&=\mu_H(X_1)+\Delta_G(g_1g_2)\mu_H(X_2). \label{eq: fibre product 2}
\end{align}
Now we consider $\nu_G(A^2)$. By the quotient integral formula,  we have
\[
 \nu_G(A^2)=\int_{G/H} r_{A^2}(g) \d \mu_{\RR}(g). 
\]
Note that $(Hg)^2=Hg^2$ as $H$ is normal, 
we have the natural lower bound of the right hand side 
\[
 \int_{G/H} r_{A^2}(g^2)\e_{\supp (r_A)}(g) \d \mu_{\RR}(g^2),
\]
and by \eqref{eq: fibre product 2} this is at least
\begin{align*}
     2\int_{G/H} r_A (g)+\Delta_G(g^2)\ell_A(g) \d\mu_\RR(g) = 2\nu_G(A)+2\int_{G/H}\Delta_G(g^2)\ell_A(g) \d\mu_\RR(g).
\end{align*}
Similarly, using the quotient integral formula we obtain 
\begin{align*}
    \mu_G(A^2)=\int_{G/H} \ell_{A^2}(g) \d \mu_{\RR}(g)\geq \int_{G/H} \ell_{A^2}(g^2)\e_{\mathrm{supp}(\ell_A)}(g) \d \mu_{\RR}(g^2).
    \end{align*}
Now    by \eqref{eq: fibre product 1} the right hand side at least
\[
    2\int_{G/H} \Delta_G(g^2)^{-1}r_A (g)+\ell_A(g) \d\mu_\RR(g)=2\mu_G(A)+2\int_{G/H} \Delta_G(g^2)^{-1}r_A (g)\d\mu_\RR(g).
\]

By applying the above lower bounds on $\nu(A^2)$ and $\mu(A^2)$, we have
\[
\left(\frac{\nu_G(A)}{\nu_G(A^2)}\right)^{\frac12}+\left(\frac{\mu_G(A)}{\mu_G(A^2)}\right)^{\frac12}\leq \left(\frac{1}{2+\frac{2\int \Delta_G(g^2)\ell_A(g)\d\mu_\RR}{\nu_G(A)}}\right)^{\frac12}+\left(\frac{1}{2+\frac{2\int \Delta_G(g^2)^{-1}r_A(g)\d\mu_\RR}{\mu_G(A)}}\right)^{\frac12}. 
\]
Using the fact that $\nu_G(A)=\int_{G/H}r_A(g)\d\mu_\RR$ and $\mu_G(A)=\int_{G/H}\ell_A(g)\d\mu_\RR$, together with \eqref{eq: relation between r and l} and  H\"older's inequality, we have
that the right hand side of the above quantity is at most
\[
\frac{\int_{G/H} \Delta_G(g)\ell_A(g)\d\mu_\RR}{\int_{G/H} \Delta_G(g)\ell_A(g)\d\mu_\RR+\int_{G/H} \Delta_G(g)^2\ell_A(g)\d\mu_\RR}+\frac{\int_{G/H} \ell_A(g)\d\mu_\RR}{\int_{G/H} \ell_A(g)\d\mu_\RR+\int_{G/H} \Delta_G(g)^{-1}\ell_A(g)\d\mu_\RR}
\]

Note that the above quantity can be written as
\begin{align}\label{eq: last line}
    & \, \bigg(\!\int_{G/H}\ell_A(g)\d\mu_\RR\int_{G/H} \Delta_G(g)\ell_A(g)\d\mu_\RR+\int_{G/H} \Delta_G(g)^{-1}\ell_A(g)\d\mu_\RR\int_{G/H} \Delta_G(g)\ell_A(g)\d\mu_\RR\\
    & +\int_{G/H} \ell_A(g)\d\mu_\RR\int_{G/H} \Delta_G(g)^2\ell_A(g)\d\mu_\RR+\int_{G/H} \ell_A(g)\d\mu_\RR\int_{G/H} \Delta_G(g)\ell_A(g)\d\mu_\RR\bigg) \nonumber\\
    \cdot&\,\bigg(\!\int_{G/H}\ell_A(g)\d\mu_\RR\int_{G/H} \Delta_G(g)\ell_A(g)\d\mu_\RR+\int_{G/H} \Delta_G(g)^{-1}\ell_A(g)\d\mu_\RR\int_{G/H} \Delta_G(g)\ell_A(g)\d\mu_\RR  \nonumber\\
   &  +\int_{G/H} \ell_A(g)\d\mu_\RR\int_{G/H} \Delta_G(g)^2\ell_A(g)\d\mu_\RR+\int_{G/H} \Delta_G(g)^{-1}\ell_A(g)\d\mu_\RR\int_{G/H} \Delta_G(g)^2\ell_A(g)\d\mu_\RR\bigg)^{-1}  \nonumber
\end{align}
Using H\"older's inequality again, we have
\begin{align}
&\,\int_{G/H}\Delta_G(g)\ell_A(g)\d\mu_\RR\int_{G/H}\ell_A(g)\d\mu_\RR \nonumber\\
    \leq&\, \left(\int_{G/H}\Delta_G(g)^2\ell_A(g)\d\mu_\RR\right)^{\frac23}\left(\int_{G/H}\Delta_G(g)^{-1}\ell_A(g)\d\mu_\RR\right)^{\frac13} \label{eq: equality ax+b}\\
    &\quad \cdot\left(\int_{G/H}\Delta_G(g)^2\ell_A(g)\d\mu_\RR\right)^{\frac13}\left(\int_{G/H}\Delta_G(g)^{-1}\ell_A(g)\d\mu_\RR\right)^{\frac23} \nonumber \\
    =&\, \int_{G/H}\Delta_G(g)^2\ell_A(g)\d\mu_\RR\int_{G/H}\Delta_G(g)^{-1}\ell_A(g)\d\mu_\RR.  \nonumber
\end{align}
By applying \eqref{eq: equality ax+b} to the second term of the second line of \eqref{eq: last line}, we conclude that
\begin{equation}\label{eq: ax+b}
\left(\frac{\nu_G(A)}{\nu_G(A^2)}\right)^{\frac12}+\left(\frac{\mu_G(A)}{\mu_G(A^2)}\right)^{\frac12}\leq 1.
\end{equation}
When the equality holds in \eqref{eq: ax+b}, equalities hold in all the inequalities in the above estimates. In particular, equality holds in \eqref{eq: equality ax+b} implies that $\Delta_G(g)$ is a constant for almost all $g\in \supp (\ell_A)$.  Since $G$ is connected and nonunimodular, $\ker\Delta_G$ has measure $0$, this would imply that $A$ has measure $0$, and hence
contradicts the fact that $A$ has positive measure.
\end{proof}

\section{Proof of the main theorems}\label{sec: proof of main theorem}

Our first theorem is the asymmetric version of Theorem~\ref{thm:main1.3}:
\begin{theorem}\label{thm: asym 1.5}
Let $G$ be a connected noncompact locally compact group, $\mu$ is a left Haar measure, and $\nu$ is a right Haar measure on $G$. Suppose $A,B$ are  compact subsets of $G$ of positive measures, and
\begin{equation}\label{eq: small exp asym}
    d(A,B^{-1})<\log K,
\end{equation}
where $d$ is the nonunimodular Ruzsa distance defined in~\eqref{eq: dist}. Then there is a continuous surjective group homomorphism $\chi: G\to L$ with compact kernel, where $L$ is a connected Lie group of dimension at most $\lfloor \log K/2 \rfloor(1+\lfloor \log K/2 \rfloor)/2$, and an open $64K^{12}$-approximate group $X$ of $G'$, and $A$ can be covered by at most $33K^{12}$ left translations of $\chi^{-1}(X)$, $B$ can be covered by at most $33K^{12}$ right translations of $\chi^{-1}(X)$. Moreover, if $A=B$ and $G$ is unimodular, with $A' =\chi(A)$, we have $
\mu_{L}((A')^2)\leq 64K^3\mu_{L}(A')
$.
\end{theorem}

\begin{proof}
As  $d(A,B^{-1})<\log K$, we have
\[
\left(\frac{\nu_G(A)}{\nu_G(AB)}\right)^{\frac{2}{\log K}}+\left(\frac{\mu_G(B)}{\mu_G(AB)}\right)^{\frac{2}{\log K}}\geq \left(K\frac{\nu_G(A)\mu_G(B)}{\nu_G(AB)\mu_G(AB)}\right)^{\frac{1}{\log K}}\geq 1,
\]
which implies that $\BM(G)\leq \lfloor \log K/2 \rfloor$. Hence by Theorem~\ref{thm: lowdimensionquotient}, there is a continuous surjective group homomorphsim $\chi: G\to L$ with compact kernel, and $L$ is a connected Lie group satisfying
\[
\dim (L)\leq \frac{\lfloor \log K/2 \rfloor(1+\lfloor \log K/2 \rfloor)}{2}.
\]

By Theorem~\ref{thm: approx groups}, there is a $64K^{12}$ open approximate group $H$ of $G$, such that $A\subseteq \Omega H$ and $B\subseteq  H\Omega$ for some finite set $\Omega$ of cardinality at most $33K^{12}$. Let $X=\chi(H)$. It is easy to see that $X$ is the desired $64K^{12}$-approximate group. When $A=B$ the desired conclusion follows from Proposition~\ref{prop: almost domi sym}.
\end{proof}

Next, we prove the nearly minimal measure expansions result:

\begin{proof}[Proof of Theorem~\ref{thm:main1}]
Using Theorem~\ref{thm: BM small}, we know that $G$ is unimodular, and moreover, there is a surjective continuous group homomorphism $\chi:G\to\RR$ with compact kernel.
We first consider the case when $\mu_G(AB)=\mu_G(A)+\mu_G(B)$. By Theorem~\ref{thm: quo domi}, there are compact sets $X,Y\subseteq\RR$ with $\mu_\RR(X+Y)=\mu_\RR(X)+\mu_\RR(Y)$, such that
\[
\mu_G(A)=\mu_\RR(X),\quad \mu_G(B)=\mu_\RR(Y),
\]
and $A\subseteq \chi^{-1}(X)$, $B\subseteq \chi^{-1}(Y)$. By Fact~\ref{fact: inverse R}, there are compact intervals $I,J\subseteq\RR$, such that $X=I$ and $Y= J$. 
Hence $A\subseteq\chi^{-1}(I)$ and $B\subseteq\chi^{-1}(J)$, this proves statement (ii).

Now we prove statement (i). By Theorem~\ref{thm: almost domi}, there are compact sets $X,Y\subseteq\RR$, with
\[
\mu_G(A\tri \chi^{-1}(X))\leq 3\dd_G(A,B),\quad \mu_G(B\tri \chi^{-1}(Y))\leq 3\dd_G(A,B),
\]
and
\[
\dd_\RR(X,Y)\leq 7\dd_G(A,B).
\]
Assume that $\dd_G(A,B)<1/20\min\{\mu_G(A),\mu_G(B)\}$, then we get
\[
\dd_\RR(X,Y)<\frac{1}{2}\min\{\mu_G(A),\mu_G(B)\}\leq \min\{\mu_\RR(X),\mu_\RR(Y)\}.
\]
Applying Fact~\ref{fact: inverse R}, we find intervals $I'',J''\subseteq \RR$ with $X\subseteq I''$ and $Y\subseteq J''$ such that 
$$\mu_\RR(I'')\le\mu_\RR(X+Y)-\mu_\RR(Y)\text{ and }\mu_\RR(J'')\le\mu_\RR(X+Y)-\mu_\RR(X).$$ In particular, we can choose subintervals $I' \subseteq I''$ and $J' \subseteq J''$ with $\mu_\RR(I')=\mu_\RR(X)$, $\mu_\RR(J')=\mu_\RR(Y)$, and
\[\mu_\RR(I'\tri X)\leq \dd_\RR(X,Y) \quad \mu_\RR(J'\tri Y)\leq \dd_\RR(X,Y).
\]

For any element $g \in A\setminus \chi^{-1}(I')$, assume that the Euclidean distance between $\chi(g)$ and $I'$ in $\RR$ is strictly greater than $50\dd_G(A,B)$.
We then have
\[
\mu_\RR(\chi(g)\chi(B)\setminus I'\chi(B))\geq 50\dd_G(A,B)-\mu_\RR(J\setminus \chi(B))\geq 40\dd_G(A,B).
\]
This implies that $\mu_G(gB\setminus \chi^{-1}(I')\chi^{-1}(J'))\geq30\dd_G(A,B)$. Therefore,
\begin{align*}
\mu_G(AB)&\geq \mu_G\big((\chi^{-1}(I')\cap A)(\chi^{-1}(J')\cap B)\big)+\mu_G(gB\setminus \chi^{-1}(I')\chi^{-1}(J'))\\
&\geq \mu_G(A)+\mu_G(B)-20\dd_G(A,B)+30\dd_G(A,B),
\end{align*}
and this is a contradiction. Thus there are intervals $I\supseteq I'$, $J\supseteq J'$ in $\RR$, such that
\[
\max\{\mu_\RR(I)-\mu_\RR(I'), \mu_\RR(J)-\mu_\RR(J')\}<100\dd_G(A,B),
\]
and $A\subseteq \chi^{-1}(I)$, $B\subseteq \chi^{-1}(J)$.
\end{proof}

Before state the next result, we first introduce the following definition. 

\begin{definition}[Interval progressions]
We say a compact set $X$ is an \emph{interval progression} on $G$ if there is a surjective continuous group homomorphism $\chi:G\to\RR$ with compact kernel, and $X=\chi^{-1}(I+P)$, where $I$ is a compact interval, and $P$ is a generalized arithmetic progression. 
\end{definition}

 \begin{theorem}[Measure expansions at most $4$ in noncompact groups]\label{thm:main2}
 Let $G$ be a connected noncompact locally compact group, $\mu_G$ a left Haar measure, and $\nu_G$ a right Haar measure on $G$. Suppose 
 $A$ is a compact subset of $G$ of positive measure, and
\begin{equation}\label{eq: expansion 4}
\left(\frac{\nu_G(A)}{\nu_G(AB)}\right)^{\frac12}+\left(\frac{\mu_G(B)}{\mu_G(AB)}\right)^{\frac12}\geq 1,
\end{equation}
Then one of the followings happens:
\begin{enumerate}[\rm (i)]
    \item There is a continuous surjective group homomorphism $\chi:G\to \RR$ with compact kernel. \begin{enumerate}[(a)]
        \item If $\mu_G(A)/20\leq\dd_G(A,A)\leq 2\mu_G(A)$, 
       $A$ is covered by an interval progression $\chi^{-1}(I+P)$, where $\dim P=O(1)$ and $\mu_\RR(I+P)=O(1)\mu_G(A)$.
       \item  If $0<\dd_G(A,A)<\mu(A)/20$, then  there is a compact interval $I\subseteq \RR$, such that $A\subseteq \chi^{-1}(I)$, and $$\mu_G(\chi^{-1}(I))-\mu_G(A)<100\dd_G(A,A).$$
    \item If $\mu_G(A^2)=2\mu_G(A)$, then there is a compact interval $I\subseteq \RR$, such that $A= \chi^{-1}(I)$.
    \end{enumerate}
    \item The equality holds in \eqref{eq: expansion 4},
and there is a continuous surjective group homomorphism $\chi:G\to H$ with compact kernel, where $H$ is either $\RR^2$, or $\RR^2 \rtimes \TT$, or $\PSL_2(\RR)$, or $\widetilde{\SL_2(\RR)}$, or $\text{Aff}(1)$. Moreover, $\chi(A)$ satisfies
\[
\mu_H(\chi(A)^2)=4\mu_H(\chi(A))
\]
where $\mu_H$ is a left Haar measure on $H$.
\end{enumerate}
 \end{theorem}

\begin{proof}
Inequality~\eqref{eq: expansion 4} implies that $\BM(G)\leq 2$. By Fact~\ref{fact: helix}, we have $\h(G)\leq \ndim(G)/3$. As a result, either $\BM(G)=1$, or $\BM(G)=2$. 
By Theorem~\ref{thm: BM small}, there is a continuous surjective group homomorphism mapping $G$ to a Lie group $N$ with compact kernel, where $N$ can only be $\RR$, $\RR^2$, $\RR^2 \rtimes \TT$, $\PSL_2(\RR)$, $\widetilde{\SL_2(\RR)}$, and the $ax+b$ group. 

We first consider the case when $\BM(G)=1$. Let $\chi:G\to\RR$ be a continuous surjective group homomorphism with compact kernel, and let $X=\chi (A)$. Suppose that
\[
\dd_G(A,A)<\frac{1}{20}\mu_G(A).
\]
Then (i) (b) and (c) follow from Theorem~\ref{thm:main1}. Let us now consider the case when $\dd_G(A,A)\geq\mu_G(A)/20$. 
 By Proposition~\ref{prop: almost domi sym}, $\mu_\RR(X+X)\leq K\mu_\RR(X)$ for some $K\leq 2^{17}$. The structural result on $X$ follows from the fact that there are arbitrarily dense subgroups of $\RR$
isomorphic to $\ZZ$. Then one can approximate any open set by a set of the form $Y + I$ where
$Y$ is contained in some such subgroup and $I$ is a  small interval in $\RR$. The structure of $Y$ comes from the Freiman theorem. More precisely, let $I(\varepsilon)$ be the interval $[-\varepsilon/2,\varepsilon/2]$.  Then there is some $\varepsilon>0$ such that
\[
\mu_\RR(X+X+I(\varepsilon))\leq 2\mu_\RR(X+X)\leq 2^{18}\mu_\RR(X)
\]
from the fact that $\mu_\RR(X+X)=\bigcap_{\varepsilon\to0}\mu_\RR(X+X+I(\varepsilon))$. We now approximate $X$ by a dense lattice as follows: we define
\[
\mathcal{X}=X+I(\varepsilon/2)\cap \frac{\varepsilon}{10}\mathbb Z. 
\]
Then we have $X\subseteq \mathcal{X}+I(\varepsilon)$ and
\[
\mathcal{X}+\mathcal{X}+I(\varepsilon/100)\subseteq X+X+I(\varepsilon). 
\]
Hence we have
\[
|\mathcal{X}+\mathcal{X}|\frac{\varepsilon}{100}\leq \mu_\RR(X+X+I(\varepsilon))\leq 2^{18}\mu_\RR(X)\leq |\mathcal{X}|2^{18}\varepsilon,
\]
and hence $|\mathcal{X}+\mathcal{X}|\leq 2^{25}|\mathcal{X}|$. By Freiman's theorem, $\mathcal{X}$ is contained in a generalized arithmetic progression $P$ and $|P|=O(1) |\mathcal{X}|$. Therefore, we have $X\subseteq P+I(\varepsilon)$. 
 This proves (i) (a). We remark that the argument is previously used by Tao in \cite[Proposition 7.1]{Tao08}. 

Finally we assume that $\BM(G)=2$. The desired conclusion follows from Theorem~\ref{thm: quo domi}. This proves (ii).
\end{proof}

\section{Effective Margulis lemma in almost-connected groups}

We will use the following proposition for finite approximate groups inside a general linear group. 
\begin{theorem}\label{thm: effective linear approx group}
There is a constant $C>0$ such that the following holds. Let $D\geq 2$ and $K\geq2$. 
    Let $A\subseteq\mathrm{GL}_D(\mathbb C)$ be a finite $K$-approximate group. Then $A$ can be covered by at most 
    \[
    K^{\exp(CD^2\log D)}
    \]
    left cosets of a nilpotent subgroup $H\leq\langle A\rangle$ of step at most $D-1$. 
\end{theorem}

We remark that a $K^{O_D(1)}$ bound is obtained by Eberhard, Murphy, Pyber, and Szab\'o~\cite{EMPS}. Theorem~\ref{thm: effective linear approx group} also makes the exponent $O_D(1)$ effective. This can be derived by using the quantitative diameter bound of  Bajpai, Dona, and Helfgott~\cite{BDH}, combinging with Jordan's theorem and the solvable case of Breuillard and Green~\cite{BG11}. The proof is recorded in a separate notes by the second author~\cite{note}. 
We remark that in this theorem as well as Theorem~\ref{thm: margulis}, $C=10^5$ suffices. 
    
The next elementary observation converts a cover of a finite set into a statement about the subgroup generated by much shorter words. This is the coset argument underlying the Margulis application in~\cite[Section~11]{BGT}.

\begin{lemma}\label{lem:cosets}
Let $G$ be a group generated by a symmetric set $T$ containing the identity, and let $H\le G$. If $T^m$ meets at most $m\geq1$ left cosets of $H$, then $[G:H]\le m$.
\end{lemma}

\begin{proof}
Consider the graph with vertices $G/H$ and edges from $gH$ to $tgH$ for $t\in T$. This graph is connected. The ball of radius $j$ about $H$ consists of the cosets represented by $T^j$.
Note that the ball of radius $j+1$ of $H$ contains at least one new vertex, unless the ball of radius $j$ about $H$ is already the entire graph. If $[G:H]>m$, the ball of radius $m$ therefore contains at least $m+1$ vertices, a contradiction.
\end{proof}

We next pass from metric covering to product growth. The same argument gives a compact neighborhood in the ambient group and a finite approximate group in the discrete subgroup.

\begin{lemma}\label{lem:packing}
Under the action and covering hypotheses of Theorem~\ref{thm: margulis}, without assuming almost connectedness, let
\[
 P=\{g\in G:d_X(x,gx)\le2r\},
\]
and let $A=P\cap\Gamma$. 
Then $P$ is a compact symmetric neighborhood of identity, and
\[
 \mu(P^2),\nu(P^2)\le K\mu(P).
\]
The set $A$ is a finite $2K$-approximate group.
\end{lemma}

\begin{proof}
Observe that the properness makes $P$ compact, and continuity makes it an identity neighborhood. Since the action is isometric, $d_X(x,g^{-1}x)=d_X(x,gx)$, so $P$ is symmetric. The triangle inequality gives $P^2x\subseteq B(x,4r)$.

For each ball in the given cover that meets $P^2 x$, choose $g_i\in P^2$ with $g_i x$ in that ball. If $g\in P^2$ and $gx$ lies in the same ball, then
\[
 d_X(x,g_i^{-1}gx)=d_X(g_ix,gx)\le2r.
\]
Thus $P^2\subseteq\bigcup_i g_iP$, with at most $K$ translates. This proves the left Haar bound. We may normalize the right Haar measure by $\nu(B)=\mu(B^{-1})$. Both $P$ and $P^2$ are symmetric, so their left and right Haar measures agree.

The set $A$ is finite because $\Gamma$ is discrete and $P$ is compact. Repeat the covering argument, choosing representatives $\gamma_i\in A^2$ only for balls meeting $A^2 x$. Then $A^2\subseteq\bigcup_i\gamma_iA$. Adding the inverses of the representatives gives a symmetric covering set of size at most $2K$.
\end{proof}

Recall that a locally compact group $G$ is \emph{almost-connected} if $G/G^0$ is compact. We will apply Theorem~\ref{thm:main1.3} for almost-connected $G$. The same proof still go through and the same conclusion holds, except that the bounded dimensional Lie quotient $L$ may not be connected. One can still from compactness to conclude $L$ has finitely many components. We first bound finite groups of automorphisms of a free abelian group. This will control the action of the components.

\begin{lemma}\label{lem:minkowski}
Every finite subgroup of $\GL_s(\mathbb Z)$ has order at most $3^{s^2}$. 
\end{lemma}

\begin{proof}
It suffices to show that reduction modulo $3$ is injective on a finite subgroup. Suppose a nonidentity finite order matrix is
congruent to the identity modulo $3$. If the order of it is not a prime number, we may take a suitable power of it, so the resulted matrix has order $p$ for some prime $p$, and still congruent to the identity modulo $3$. We write it as $I+3^jB$, where $j\ge1$ and $B$ is nonzero modulo $3$, and $(I+3^jB)^p = I$. 

Expansion of $(I+3^jB)^p=I$ modulo $3^{j+1}$ forces $p=3$. Expanding the cube gives
\[
 B+3^jB^2+3^{2j-1}B^3=0.
\]
Modulo $3$ gives $B\equiv 0\pmod 3$ which is a contradiction. 
The finite subgroup therefore embeds in $\GL_s(\mathbb F_3)$, whose order is at most $3^{s^2}$.
\end{proof}

For $H\leq G$ we use $C_G(H)$ to denote centralizer of $H$ in $G$, and $Z(H)=C_H(H)$. 

\begin{lemma}\label{lem:adjoint kernel}
Let $L$ be an $d$-dimensional Lie group with finitely many components and no nontrivial compact normal subgroup. Then there are integers $a,b$ with $a+b\leq d$ such that
\[
 C_L(L^0)\cong\RR^a\times\ZZ^b. 
\]
If $\Lambda\le L$ is discrete and $D=\Lambda\cap C_L(L^0)$, then $D$ is free abelian of rank at most $d$, and
\[
 [\Lambda:C_\Lambda(D)]\le3^{d^2}.
\]
\end{lemma}

\begin{proof}
An automorphism of a connected Lie group is determined by its
differential, so $\ker\Ad_L=C_L(L^0)$. Write $Z=Z(L^0)$ then $C_L(L^0)\cap L^0=Z$, and hence $Z$ is central in $C_L(L^0)$.
As $L/L^0$ is finite, we have $C_L(L^0)/Z$ is finite. Since $Z\leq Z(C_{L}(L^0))$, Schur's theorem therefore makes $[C_{L}(L^0),C_{L}(L^0)]$ finite.
Since $C_{L}(L^0)$ is characteristic in $L$, this finite subgroup is normal in $L$. It must be trivial. Thus $C_{L}(L^0)$ is abelian. We write $C=C_{L}(L^0)$. 

Let $a=\dim C^0=\dim Z^0$. The homotopy sequence of the quotient map $L^0\to L^0/Z$ gives a surjection
\[
 \pi_1(L^0/Z)\to\pi_0(Z).
\]
The fundamental group of a connected Lie group is finitely
generated abelian and has free rank at most the dimension of the
group. Since $\dim(L^0/Z)=d-a$, the group $\pi_0(Z)$ is finitely generated of rank at most $d-a$.

As $C\leq L$, $C^0\subseteq C\cap L^0$ and hence $C^0=Z^0$. 
Hence $\pi_0(C)$ is finitely generated, with the same free rank
$b\leq d-a$ as $\pi_0(Z)$. The structure of abelian Lie groups gives
\[
 C\cong\RR^u\times\ZZ^b\times T,
\]
where $T$ is compact. The subgroup corresponding to $T$ is the unique maximal compact subgroup of $C$, so it is characteristic in $C$ and normal in $L$. Thus $T$ is trivial, we have $u=a$, and $a+b\le d$.

Finally, let us project the discrete subgroup $D$ to $\ZZ^b$. Its image has rank at most $b$, and its kernel is a discrete subgroup of $\RR^a$, hence a lattice of rank at most $a$. Therefore $D$ is free abelian of rank $s\le a+b\le d$.

Observe that $D$ is normal in $\Lambda$, conjugations on $\Lambda$ defines a homomorphism 
\[
\rho:\Lambda\to\mathrm{Aut}(D)\cong \mathrm{GL}_s(\mathbb Z)
\]
and $\ker(\rho) = C_\Lambda(D)$.
By the finiteness of $L/L^0$, the image of $\rho$ is a finite subgroup of $\GL_s(\mathbb Z)$. Lemma~\ref{lem:minkowski} then bounds its order by $3^{s^2}\le3^{d^2}$.
\end{proof}

Finally we are ready to prove the main theorem of the section. 

\begin{proof}[Proof of Theorem~\ref{thm: margulis}]
    Apply Lemma~\ref{lem:packing} and define 
\[
 P=\{g\in G:d_X(x,gx)\le2r\},
\]
let $A=P\cap\Gamma$.
Let $\pi:G\to L=G/H$ be the quotient from Theorem~\ref{thm:main1.3} for almost-connected group $G$ and hence $L$ is a Lie group with $L/L^0$ finite. In particular, $D=\dim L\le (\log K)^2$.
Let $\chi=\Ad_L\circ\pi$ and $\chi$ is a homomorphism into $\GL_{D}(\mathbb C)$. By Theorem~\ref{thm: effective linear approx group}, $\chi(A)$ can be covered by at most $L(D,K)\leq M$ left cosets of a nilpotent subgroup $N$ of step at most $D-1$.

Write $S=S_{r/2M}(x)$ and $\Delta=\langle S\rangle$. Then every word of length at most $M$ in $S$ moves $x$ by less than $Mr/2M<r$. Thus $S^{M}\subseteq A$. Apply Lemma~\ref{lem:cosets} (with $T=\chi(S)$) to
\[
V=\chi(\Delta)\cap N.
\]
A left coset of $N$ meets $\chi(\Delta)$ either in the empty set or in a left coset of $V$. Hence, with $E_0=\Delta\cap \chi^{-1}(N)$ we have
\[
[\Delta:E_0]\le M.
\]

We next lift this nilpotent image. Since $H$ is compact, the quotient map $\pi$ is proper. It follows that $\Lambda=\pi(\Delta)$ is discrete. Indeed, the inverse image of a compact set meets $\Delta$ in finitely many points. Let $D=\Lambda\cap\ker\Ad_L$ and
\[
 E=E_0\cap\pi^{-1}\bigl(C_\Lambda(D)\bigr).
\]
Lemma~\ref{lem:adjoint kernel} gives $[\Delta:E]\le M3^{D^2}$. The adjoint kernel of $\pi(E)$ is central in $\pi(E)$, while its adjoint image lies in $N$. Therefore $\pi(E)$ is nilpotent of class at most $D$.

Finally, let $F=\Delta\cap H$. This subgroup is finite, normal in $\Delta$, and contained in $E$. The quotient $E/F$ is isomorphic to $\pi(E)$. To obtain virtual nilpotence, observe that the index of $C_E(F)$ in $E$ is at most $|\Aut(F)|$. The kernel of the map $C_E(F)\to E/F$ is $Z(F)$, which is central in $C_E(F)$. Thus $C_E(F)$ is nilpotent of step at most $D+1$. If $\Gamma$ is torsion-free, then $F=1$ and $E$ already has the desired nilpotency property.
\end{proof}

\appendix

\section{Nonunimodular Ruzsa distance and open approximate groups}

In this section, we introduce a generalization of Ruzsa's  distance in nonunimodular groups, which will be used in the next section. Throughout the section, $G$ is a (not necessarily connected) locally compact group, $\mu_G$ is a left Haar measure, and $\nu_G$ is a right Haar measure on $G$. We will also assume that $\nu_G = \mu^{-1}_G $, but this assumption will not be needed until after Lemma~\ref{lem: ruzsa covering}. Suppose $A,B$ are compact subsets of $G$ of positive measures. Let
\begin{equation}\label{eq: dist}
d(A,B)=\log\frac{\nu_G(AB^{-1})\mu_G(AB^{-1})}{\nu_G(A)\mu_G(B^{-1})}.
\end{equation}
It is easy to see that $d(A,A)=0$ if and only if $A$ is a right translation of a subgroup of $G$. Thus for most of the sets $A$, we will have $d(A,A)>0$.
On the other hand, the distance still behaves like a metric in many aspects:
\begin{lemma}\label{lem: ruzsa}
Let $d$ be as in \eqref{eq: dist}, and $A,B,C$ are compact subsets of $G$ with positive measures. Then we have
\begin{enumerate}[\rm (i)]
    \item \emph{(Nonnegativity)} $d(A,B)\geq 0$;
    \item \emph{(Symmetric)} $d(A,B)=d(B,A)$;
    \item \emph{(Left translation invariance)} $d(A,B)=d(aA,bB)$ for every $a,b\in G$;
    \item \emph{(Triangle inequality)} $d(A,B)\leq d(A,C)+d(C,B)$.
\end{enumerate}
\end{lemma}
Note that when $G$ is unimodular, by ignoring a constant factor, \eqref{eq: dist} becomes
\[
d_R(A,B)=\log\frac{\mu_G(AB^{-1})}{\sqrt{\mu_G(A)\mu_G(B^{-1})}},
\]
which is known as Ruzsa's distance. Lemma~\ref{lem: ruzsa} for Ruzsa's distance was first established by Ruzsa~\cite{ruzsa} for discrete groups, and extended by Tao~\cite{Tao08} to unimodular groups. Our proof below is similar as the one in~\cite{ruzsa,Tao08}.
\begin{proof}[Proof of Lemma~\ref{lem: ruzsa}]
Statement (i) is immediate from the definition. Statement (ii) is also clear in the special case where $\nu_G= \mu_G^{-1}$; the more general case follows from the uniqueness of Haar measure. Next, observe that
\begin{align*}
    d(aA,bB)&=\log\frac{\nu_G(aAB^{-1})\mu_G(AB^{-1}b^{-1})}{\nu_G(aA)\mu_G(B^{-1}b^{-1})}\\
    &=\log\frac{\Delta_G(a)^{-1}\nu_G(AB^{-1})\Delta_G(b)^{-1}\mu_G(AB^{-1})}{\Delta_G(a)^{-1}\nu_G(A)\Delta_G(b)^{-1}\mu_G(B^{-1})}=d(A,B),
\end{align*}
which proves (iii). Finally, to show the triangle inequality, it suffices to show that
\begin{equation}\label{eq: triangle}
\nu_G(AB^{-1})\mu_G(AB^{-1})\mu_G(C)\nu_G(C)\leq \nu_G(AC^{-1})\mu_G(AC^{-1})\nu_G(CB^{-1})\mu_G(CB^{-1}).
\end{equation}
Using the right translation invariance of $\nu_G$ we get
\begin{align*}
\nu_G(AC^{-1})\nu_G(CB^{-1})&=\iint_{G\times G}\e_{AC^{-1}}(x)\e_{CB^{-1}}(y)\d\nu_G(x)\d\nu_G(y)\\
&= \int_G\left(\int_G \e_{AC^{-1}}(zy^{-1})\e_{CB^{-1}}(y)\d\nu_G(y)\right)\d\nu_G(z).
\end{align*}
Note that when $z$ is in $AB^{-1}$, there are $a\in A$ and $b\in B$ such that $z=ab^{-1}$. When $y\in Cb^{-1}$, clearly $\e_{AC^{-1}}(zy^{-1})\e_{CB^{-1}}(y)=1$. As $\nu_G(Cb^{-1})=\nu_G(C)$, we have
\[
\nu_G(AC^{-1})\nu_G(CB^{-1})\geq \int_{ G}\e_{AB^{-1}}(z)\nu_G(C)\d\nu_G(z)=\nu_G(AB^{-1})\nu_G(C).
\]
Likewise, $\mu_G(AC^{-1})\mu_G(CB^{-1})\geq\mu_G(AB^{-1})\mu_G(C)$, and this proves~\eqref{eq: triangle}.
\end{proof}

We will need the following nonunimodular version of Ruzsa's covering lemma:
\begin{lemma}\label{lem: ruzsa covering}
Suppose $A,B$ are compact subsets of $G$. Then we have the following:
\begin{enumerate}[\rm (i)]
    \item If $\nu_G(AB)\leq K\nu_G(A)$, then there exists a finite subset $\Omega\subseteq B$ of cardinality at most $K$ such that $B\subseteq A^{-1}A\Omega$.
    \item If $\mu_G(AB)\leq K\mu_G(B)$, then there exists a finite subset $\Omega\subseteq A$ of cardinality at most $K$ such that $A\subseteq \Omega BB^{-1}$.
\end{enumerate}
\end{lemma}
\begin{proof}
We will only prove (i), as the proof of (ii) is the same. Let $\Omega$ be a maximal subset of $B$ such that right translations of $A$ by elements in $\Omega$ are pairwise disjoint. Thus by the right-translation invariance of $\nu_G$, $|\Omega|\leq K$. By the maximality of $\Omega$, for every $y\in B$, there is $y_0\in\Omega$ such that $Ay\cap Ay_0\neq\varnothing$. This implies $B\subseteq A^{-1}A\Omega$.
\end{proof}

Given two compactly supported functions $f$ and $g$ on $G$, recall that the convolution $f*g$ is defined as
\[
f*g (x)=\int_G f(y)g(y^{-1}x)\d\mu_G(y).
\]
When $G$ is nonunimodular, in general we do not have $f*g (x)=\int_G f(xy^{-1})g(y)\d\mu_G(y)$. Using the fact that $\nu_G=\mu_G^{-1}$, one can easily show that
\[
f*g (x)=\int_G f(xy^{-1})g(y)\d\nu_G(y).
\]
 We define the (right) multiplicative energy $E(A,B)=\int( \e_A*\e_B)^2\d\nu_G$. The following lemma gives us the relation between $E(A,A^{-1})$ and $E(A^{-1},A)$.
\begin{lemma}\label{lem: E(A,A^{-1})}
    Let $A$ be a compact subset of $G$. Then 
    \[
    E(A^{-1},A)\geq \max_{a\in A}\frac{1}{\Delta(a)}E(A,A^{-1}). 
    \]
\end{lemma}
\begin{proof}
   We use  $\widetilde{f}(x)$ to denote $f(x^{-1})$. Then from the definition we have
   \[
   E(A^{-1},A)=\int_G(\e_{A^{-1}}*\e_{A}(x))^2\d\nu_G(x)=\widetilde{\e_{A^{-1}}*\e_{A}}*\e_{A^{-1}}*\e_{A}(1). 
   \]
   Using the fact that $\widetilde{f*g}=\widetilde{g}*\widetilde{f}$, the right hand side of the above equality can be rewritten as 
   \[
\widetilde{\e_{A}}*\widetilde{\e_{A^{-1}}}*\e_{A^{-1}}*\e_A(1),
   \]
   which is the same as
   \[
   \e_{A^{-1}}*\e_A*\e_{A^{-1}}*\e_{A}(1). 
   \]
We now use the definition of the convolution, and the above quantity is equal to
\[
\int_G \e_{A^{-1}}(y^{-1})(\e_A*\e_{A^{-1}}*\e_{A}(y))\d\nu_G(y). 
\]
By Fact~\ref{fact: 2.2} (iv), the above quantity can be rewritten as
\[
\int_G \e_{A^{-1}}(y^{-1})\e_A*\e_{A^{-1}}*\e_{A}(y)\frac{1}{\Delta(y)}\d\mu_G(y).
\]
Hence we have
\[
E(A^{-1},A)\geq \max_{a\in A}\frac{1}{\Delta(a)} \e_A*\e_{A^{-1}}*\e_{A}*\e_{A^{-1}}(1). 
\]

On the other hand, by a similar argument, we have
\[
E(A,A^{-1})= \widetilde{\e_{A}*\e_{A^{-1}}}*\e_{A}*\e_{A^{-1}}(1)=\e_A*\e_{A^{-1}}*\e_{A}*\e_{A^{-1}}(1),
\]
and the lemma follows. 
\end{proof}
As a corollary of Lemma~\ref{lem: E(A,A^{-1})}, one can see that when $G$ is \emph{unimodular}, we have 
\[
E(A,A^{-1})=E(A^{-1},A). 
\]

\begin{definition}[Approximate groups]
An open and precompact set $A\subseteq G$ is a \emph{$K$-approximate group}, if $A=A^{-1}$, $\mathrm{id}_G\in A$, and $A^2\subseteq \Omega A$ for some finite set $\Omega$ of cardinality at most $K$.
\end{definition}

We define the (right-translation) approximate measure stabilizer
\[
\Stab_{\nu_G}^{\varepsilon}(A)=\{g\in G: \nu_G(A\setminus Ag)<\varepsilon\},
\]
where $\varepsilon < \nu_G(A)$.
We have the following simple observation, and the proof is straightforward.
\begin{fact}
If $A$ is compact, then $S=\Stab_G^{\varepsilon}(A)$ contains $1_G$ and is open, precompact, and equal to $S^{-1}$.
\end{fact}

The following proposition is a nonunimodular analog of~\cite[Proposition 4.5]{Tao08}. The proof is similar as the one in~\cite{Tao08}, with certain modifications for the nonunimodular groups.
\begin{proposition}\label{prop: Tao}
Suppose $A$ is a compact subset of $G$ with $\nu_G(AA^{-1})\leq K\mu_G(A)$ and $\max_{a\in A}\Delta(a)=1$.  Let
\[
S= \Stab_{\nu_G}^{(2K-1)\nu_G(A)/2K}(A).
\]
Then we have the following:
\begin{enumerate}[\rm (i)]
    \item $S$ is symmetric and $S\subseteq A^{-1}A$;
    \item $\nu_G(S)\geq\mu_G(A)/2K$;
    \item for all integers $n\geq 1$, $\nu_G(AS^{n}A^{-1})\leq 2^nK^{2n+1}\mu_G(A)$.
\end{enumerate}
\end{proposition}
\begin{proof}
(i) is straightforward by the way we construct $S$. Indeed, for every $s\in S$, $A\cap As\neq\varnothing$ and hence $S\subseteq A^{-1}A$. 
By the Cauchy--Schwarz inequality, we have
\[
E(A^{-1},A)\geq E(A,A^{-1})\geq\frac{\nu_G(A)^2\mu_G(A)^2}{\nu_G(AA^{-1})}\geq \frac{\nu_G(A)^2\mu_G(A)}{K}.
\]
Also by the choice of $S$, we have
\begin{align*}
\int_{G\setminus S}(\e_{A^{-1}}*\e_A(x))^2\d\nu_G(x)\leq \frac{\nu_G(A)}{2K}\int_G \e_{A^{-1}}*\e_A(x)\d\nu_G(x)  \leq    \frac{\nu_G(A)^2\mu_G(A)}{2K}.
\end{align*}
Thus we have
\[
\int_{S}(\e_{A^{-1}}*\e_A(x))^2\d\nu_G(x)\geq \frac{\nu_G(A)^2\mu_G(A)}{2K}.
\]
This gives us $\nu_G(S)\geq \mu_G(A)/2K$, which proves (ii).

Next we prove statement (iii). Note that we have
\begin{align}
&\,\int_{G^{n+1}}\e_{AS^nA^{-1}}(y_0\cdots y_n)\prod_{i=0}^n \big(\e_{AA^{-1}}(y_i)\d\nu_G(y_i) \big)\label{eq: left}\\
=&\,\int_{G^{n+1}}\e_{AS^nA^{-1}}(x)\e_{AA^{-1}}(xy_n^{-1}\cdots y_1^{-1})\prod_{i=1}^n\big(\e_{AA^{-1}}(y_i) \d\nu_G(y_i)\big)\d\nu_G(x). \label{eq: right}
\end{align}
Fix an arbitrary $x\in AS^nA^{-1}$. Then we can fix $a,b_n\in A$ and $s_1,\dots,s_n\in S$ such that $x=as_1\cdots s_nb_n^{-1}$. For $b_0,\dots,b_{n-1}\in G$, if we set $y_i=b_{i-1}s_ib_i^{-1}$ for $1\leq i\leq n$, then $xy_n^{-1}\cdots y_1^{-1}=ab_0^{-1}$. Thus \begin{align*}
&\,\int_{G^{n}}\e_{AA^{-1}}(xy_n^{-1}\cdots y_1^{-1})\prod_{i=1}^n\big(\e_{AA^{-1}}(y_i) \d\nu_G(y_i)\big)\\
=&\, \int_{G^{n}}\e_{AA^{-1}}(ab_0^{-1})\prod_{i=1}^n \big(\e_{AA^{-1}}(b_{i-1}s_ib_i^{-1})\d\nu_G(b_{i-1}) \big).
\end{align*}
The above integrand is $1$ if $b_i\in A$ and $b_{i}s_{i+1}\in A$ for $0\leq i\leq n-1$. Hence, the above integral is at least
\[
\int_{G^n}\prod_{i=0}^{n-1}\e_A(b_i)\e_A(b_is_{i+1})\d\nu_G(b_i)=\prod_{i=0}^{n-1}\nu_G(A\cap As_{i+1}^{-1})\geq \left(\frac{\nu_G(A)}{2K}\right)^n
\]
as $s_i$ are from the approximate measure stabilizer $S$. Thus, \eqref{eq: right} is at least $$\nu_G(AS^nA^{-1})(\nu_G(A)/2K)^n.$$ On the other hand, \eqref{eq: left} is bounded from above by
\[
\nu_G(AA^{-1})^{n+1}\leq K^{n+1}\mu_G(A)^{n+1}.
\]
Note that as we have
\[
\nu_G(A)=\int_G\e_A(x)\frac{1}{\Delta_G(x)}\mu_G(A)\geq \mu_G(A),
\]
so $\nu_G(AA^{-1})^{n+1}\leq K^{n+1}\mu_G(A)\nu_G(A)^{n}$. This proves the proposition.
\end{proof}

The next theorem says that sets with small Ruzsa's distance are commensurable to approximate groups. This is the nonunimodular counterpart of~\cite[Theorem 4.6]{Tao08}.

\begin{theorem}\label{thm: approx groups}
Suppose $A,B$ are compact subsets of $G$, and $d(A,B^{-1})\leq \log K$. Then there is $H\subseteq G$ satisfying:
\begin{enumerate}[(i)]
    \item (group-like) $H$ is a $64K^{12}$-approximate group;
    \item (smallness) $\mu_G(H)=\nu_G(H)\leq 4K^5(\nu_G(A)\mu_G(B))^\frac12$;
    \item (largeness) $A\subseteq \Omega H$ and $B\subseteq H\Omega$ where $|\Omega|\leq 33K^{12}$.
\end{enumerate}
\end{theorem}
\begin{proof}
Choose $g_1\in A$, and $g_2\in B$ such that $\Delta_G(g_1)=\max_{a\in A}\Delta_G(a)$. Replacing $A$ with $g_1^{-1}A$ and $B$ with $Bg_2^{-1}$ if necessary, we can arrange that  $\mathrm{id}_G\in A\cap B$ and $\max_{a\in A}\Delta(a)=1$.

 By Lemma~\ref{lem: ruzsa}, \[
d(A,A)\leq 2d(A,B^{-1})\leq 2\log K,
\]
that is $\nu_G(AA^{-1})\leq K\mu_G(A)$ since we chose $\nu_G=\mu_G^{-1}$. Let
\[
S= \Stab_{\nu_G}^{(2K-1)\nu_G(A)/2K}(A).
\]
Note that $S^n\subseteq AS^nA^{-1}$. By Proposition~\ref{prop: Tao}, we have
\[
\nu_G(S^5)\leq\nu_G(AS^5A^{-1})\leq 32K^{11}\mu_G(A)\leq 64K^{12}\nu_G(S).
\]
As $S$ is symmetric, by Lemma~\ref{lem: ruzsa covering}, $S^2$ is a $64K^{12}$-approximate group. Let $H=S^2$. Also by Proposition~\ref{prop: Tao} again we have
\[
\nu_G(H)\leq \nu_G(AHA^{-1})\leq 4K^5\nu_G(A). 
\]
By Proposition~\ref{prop: Tao} again and the fact that $\mathrm{id}_G\in A$,
\[
\nu_G(HA^{-1})\leq\nu_G(AHA^{-1})\leq 4K^5\mu_G(A)\leq 8K^6\nu_G(S)\leq 8K^6\nu_G(H).
\]
Hence by Lemma~\ref{lem: ruzsa covering}, there is $\Omega_1\subseteq A$ of cardinality at most $8K^6$ such that $A\subseteq \Omega_1 H$.

By Lemma~\ref{lem: ruzsa} again, $d(B^{-1},H)\leq d(A,B^{-1})+d(A,H)$, and
\[
d(A,H)=\log\frac{\nu_G(AH)\mu_G(AH)}{\nu_G(A)\mu_G(H)}\leq \log \frac{4K^5\nu_G(A)8K^6\mu_G(H)}{\nu_G(A)\mu_G(H)}=\log 32K^{11}.
\]
This implies that
\[
\nu_G(HB)\leq 32K^{12}\nu_G(H)\frac{\mu_G(B)}{\mu_G(HB)}\leq 32K^{12}\nu_G(H).
\]
Thus by Lemma~\ref{lem: ruzsa covering} again, there is $\Omega_2\subseteq B$ of cardinality at most $32K^{12}$ such that $B\subseteq H\Omega_2$. Set $\Omega=g_1^{-1}\Omega_1\cup \Omega_2g_2^{-1}$. Note that by Kemperman's inequality, we may assume $K\geq 4$, and hence $|\Omega|\leq 33K^{12}$ which finishes the proof.
\end{proof}

\bibliographystyle{amsalpha}
\bibliography{reference}

\end{document}